\documentclass[11pt,a4paper,reqno]{amsart}

\usepackage[lmargin=1.4 in, rmargin= 1.4in, tmargin=1.7in]{geometry}
\usepackage{quiver}
\usepackage{marginnote} 

\title{Hochschild cohomology and extensions of triangulated categories}

\author{Alessandro Lehmann}
\address[Alessandro Lehmann]{Universiteit Antwerpen, Departement Wiskunde, Middelheimcampus,
Middelheimlaan 1,
2020 Antwerp, Belgium}
\email{alessandro.lehmann@uantwerpen.be}

\address{SISSA, Via Bonomea 265, 34136 Trieste TS, Italy}
\email{alehmann@sissa.it}

\author{Wendy Lowen} 
\address[Wendy Lowen]{Universiteit Antwerpen, Departement Wiskunde, Middelheimcampus,
Middelheimlaan 1,
2020 Antwerp, Belgium}
\email{wendy.lowen@uantwerpen.be}

\thanks{
This project has received funding from the European Research Council (ERC) under the European Union’s Horizon 2020 research and innovation programme (grant agreement No. 817762).
}

\makeatletter
\@namedef{subjclassname@2020}{
  \textup{2020} Mathematics Subject Classification}
\makeatother

\subjclass[2020]{18G80, 16E40 (Primary), 	18G70, 18N25 (Secondary)}

\keywords{}

\usepackage{graphicx}
\usepackage[english]{babel}
\usepackage{amsmath}
\usepackage{amssymb}
\usepackage{amsthm}
\usepackage{comment}
\usepackage[all,cmtip]{xy}
\usepackage{cancel}
\usepackage[alpine]{ifsym}
\usepackage{enumerate}
\usepackage{stmaryrd}
\usepackage{tikz}
\usepackage{tikz-cd}
\usetikzlibrary{matrix}
\usepackage[toc,page]{appendix}
\usepackage[hidelinks]{hyperref}
\usepackage{csquotes}

\usepackage{bbm}

\DeclareMathOperator{\Hom}{Hom}
\DeclareMathOperator{\id}{id}

\DeclareMathOperator{\Fun}{Fun}

\usepackage[
backend=biber,
style=alphabetic,
sorting=nyt,
doi=true,
url=false,
maxcitenames=50,
maxnames=50,
]{biblatex}
\newcommand{\op}[1]{{#1}^{\mathrm{op}}}

\newcommand{\Hm}[1]{\Hom_{#1}}

\newcommand{\HH}{\operatorname{HH}}

\newcommand{\tow}[1]{\overset{#1}{\to}}

\newcommand{\T}{\mathcal{T}} 
\newcommand{\A}{\mathcal{A}} 
\newcommand{\B}{\mathcal{B}} 
\newcommand{\C}{\mathcal{C}} 
\newcommand{\I}{E} 

\newcommand{\cdef}{\operatorname{cDef}}
\newcommand{\CatDef}{\operatorname{CatDef}}
\newcommand{\ke}{k[\varepsilon]}
\newcommand{\Tw}{\operatorname{\Tw}}

\newcommand{\Ker}{\operatorname{Ker}}

\newcommand{\Img}{\operatorname{Im}}
\newcommand{\Coker}{\operatorname{Coker}}

\DeclareMathOperator{\Ext}{Ext}

\newcommand{\Mod}{\operatorname{-Mod}}

\usepackage[textsize=scriptsize]{todonotes}

\newtheorem{Thm}{Theorem}[section]
\newtheorem*{Thm*}{Theorem}
\newtheorem{Lem}[Thm]{Lemma}
\newtheorem{Prop}[Thm]{Proposition}
\newtheorem{Cor}[Thm]{Corollary}

\theoremstyle{definition}
\newtheorem{Def}[Thm]{Definition}
\newtheorem{Ex}[Thm]{Example}

\theoremstyle{remark}
\newtheorem{Rem}[Thm]{Remark}

\begin{document}

\begin{abstract}
We define a notion of categorical first order deformations for (enhanced) triangulated categories. For a category $\T$, we show that there is a bijection between $\HH^2(\T)$ and the set of categorical deformations of $\T$. We show that in the case of curved deformations of dg algebras considered in \cite{MoritacDef}, the $1$-derived category of the deformation (introduced in \cite{Nder}) is a categorical deformation of the derived category of the base; the Hochschild class identified by this deformation is shown to restrict to the class defining the deformation of the algebra. As an application, we give a conceptual proof of the fact that (for a smooth base) the filtered derived category of a dg deformation yields a categorical resolution of the classical derived category.
\end{abstract}
\maketitle
\vspace{-10mm}

\begingroup
\small
\tableofcontents
\endgroup
\section{Introduction}

It is hard to overestimate the importance of triangulated categories in subjects ranging from algebraic geometry and representation theory to mathematical physics. In practice, one usually works with enhanced versions of these categories in order to be able to compute invariants and make various constructions, and we will be using the term triangulated category to signify a pretriangulated dg category. From the point of view of noncommutative algebraic geometry, one likes to view these categories as algebro-geometric objects in their own right. As such, it is surprising that no comprehensive deformation theory currently exists. 

Based upon the picture of Homological Mirror Symmetry due to Kontsevich, such a deformation theory is expected to be governed by Hochschild cohomology \cite{KontHMS}.
However, already in the compactly generated case, where categories are determined by a single algebraic generator, deformation theory is known to hit the curvature problem \cite{MoritaDef, keller2009nonvanishing, LowenCurvature}. This problem can be rephrased in the setting from \cite{DAGX} by saying that the elementary deformation theory of dg (or $A_\infty$-)categories does not constitute, in general, a formal moduli problem (FMP); its ``best approximation'' as a FMP corresponds, via the celebrated correspondence between deformation problems and DGLAs, to the Hochschild complex. Several partial solutions have been proposed to this problem, always using the fact that, under certain hypotheses, many Hochschild classes can be represented by an uncurved object; see for example \cite{LowenCurvature, Blanc_Katzarkov_Pandit_2018, HaidenCY} for the case of formal deformations of appropriately right-bounded objects, \cite{LOWEN2013441} for the case of schemes and \cite{Tst3} for deformations of dg categories with a nice $t$-structure. Still, for an arbitrary base it has long been known \cite{MoritaDef, DAGX} that uncurved deformations do not suffice to span the whole Hochschild cohomology. In \cite{MoritacDef}, the first author proposed a solution to this problem for the case of a dg algebra, which incorporates the naturally occurring curved deformations; this makes use of certain filtered derived categories, introduced by the authors in \cite{Nder}.

The main goal of the present paper is to give a direct deformation theoretical interpretation of the entire second Hochschild cohomology group of an arbitrary triangulated category. This passes through defining a notion of \emph{categorical deformation} of a triangulated category. As our first main result, we show that there is a bijection $\mu_\T$ between the set of categorical deformations of a triangulated category $\T$ and its second Hochschild cohomology group (Theorem \ref{HH2Defo}). This construction recovers the filtered derived category from \cite{Nder} in the compactly generated case: for a dg algebra $A$ we obtain a commutative square (Theorem \ref{mainth})
\begin{equation}\label{squarec}\begin{tikzcd}
	{\cdef_A(\ke)} & {\HH^2(A)} \\
	{\CatDef_{D(A)}(\ke)} & {\HH^2(D(A))}
	\arrow["\nu", from=1-1, to=1-2]
	\arrow["{D^\varepsilon(-)}"', from=1-1, to=2-1]
	\arrow["{\chi_A}", from=1-2, to=2-2]
	\arrow["\mu_{D(A)}", from=2-1, to=2-2]
\end{tikzcd}
\end{equation}
in which $\cdef_A(\ke)$ is the set of first order curved Morita deformations of $A$, $\CatDef_{D(A)}(\ke)$ is the set of (first order) categorical deformations of $D(A)$,  $\chi_A$ is the characteristic morphism from \cite{lowchar}, $\nu$ is the bijection from \cite{MoritacDef}, and the map $D^\varepsilon(-)$ associates to a curved Morita deformation its filtered derived category.

It should be pointed out that there is a qualitative difference between the \emph{curved deformations} of the base algebra $A$ -- studied in \cite{MoritacDef} -- and the \emph{categorical deformations} of its derived category $D(A)$ that we propose here; indeed, while the base algebra $A$ might have deformations that are curved -- and thus live outside of the realm of classical homological algebra -- a categorical deformation of $D(A)$ is \emph{always} a triangulated category -- by definition. This is one of the main novelties of the present paper: given an arbitrary Hochschild cocycle, even one that has nontrivial curvature component, our approach allows to construct a deformation which is again a dg category. This situation makes sense from the perspective of noncommutative algebraic geometry; a noncommutative space can be represented by different algebraic models: small ones -- for example, a dg algebra -- and large ones -- for example, its derived category. The point is that small and large models have qualitatively different deformation theories. While for a small model there is no escaping curvature -- there exist algebra deformations that are intrinsically curved, see \cite{MoritaDef} -- the large models offer enough flexibility to realize every deformation without leaving the world of triangulated categories. To obtain this, however, one must part with the idea that small models and large models should correspond via classical (or, for that matter, second kind) derived categories.
Rather, as we will explain next, our novel categorical deformations can be interpreted as a kind of categorification of the classical square zero extension 
\[
0 \rightarrow A \rightarrow A[\varepsilon] \rightarrow A \rightarrow 0
\]
associated to a first order algebra deformation.

In the first part of the paper, we define categorical deformations of a triangulated category $\T$ (\S \ref{defsec}) and their equivalences (\S \ref{equisec}), leading to the set
of categorical deformations up to equivalence featuring in the lower left corner of \eqref{squarec}.
We further construct the morphism $\mu_\T$ (\S \ref{equisec}) as well as its inverse (\S \ref{inverse}), establishing the bijection in \S \ref{proofof}.

Roughly speaking, a categorical deformation of $\T$ consists of a triangulated category $\T_\varepsilon$ together with two “extensions”:
a recollement
 \begin{equation}\label{recoll}
 \begin{tikzcd}
	\T & {\T_\varepsilon} & \T
	\arrow["i", hook, from=1-1, to=1-2]
	\arrow["\I", from=1-2, to=1-3]
	\arrow["K", curve={height=-22pt}, from=1-2, to=1-1]
	\arrow["Q"', curve={height=22pt}, from=1-2, to=1-1]
	\arrow["G"', curve={height=22pt}, from=1-3, to=1-2]
	\arrow[curve={height=-22pt}, from=1-3, to=1-2]
\end{tikzcd}    
 \end{equation}
together with a Yoneda 2-extension of functors that can be informally depicted as

\[
0\to \I \tow{\delta_1}K \tow{\alpha} Q \tow{\delta_2} \I \to 0.
\]

Briefly put, the Hochschild class corresponding to such a deformation is obtained by composing the latter with the functor $G$ and taking the associated class in \[
\Ext^2_{\Fun(\T, \T)}(\I G, \I G)\cong \Ext^2_{\Fun(\T, \T)}(\id_\T, \id_\T)\cong \HH^2(\T).
\]

Conversely, for an arbitrary Hochschild class in $\HH^2(\T)$ represented by \[
\id_\T[-1]\to \id_\T[1]
\] the corresponding deformation $\T_\varepsilon$ is constructed by gluing -- in the sense of \cite{Kuznetsov_2014} -- two copies of $\T$ along the cone of the 2-class. Note however that we are in a technically slightly more challenging setup than \cite{Kuznetsov_2014}, with the arrows in \eqref{recoll} depicting quasi-functors rather than actual dg functors.

In the second part of the paper, we prove the promised compatibility with the filtered derived categories from \cite{Nder}. To explain this in some detail, let $A$ be a dg algebra and consider the $1$-derived category $D^\varepsilon(A_\varepsilon)$ of the curved deformation $A_\varepsilon$ associated to a given Hochschild $2$-class of $A$. It is not hard to show that this naturally determines a categorical deformation of $D(A)$ -- indeed the notion we put forth is an abstraction of the recollement observed in \cite{Nder}.
The main difficulty in proving commutativity of \eqref{squarec} is finding an appropriate description of the functor $G$. Interestingly, $G$ is not a dg functor but has a nice description as an $A_\infty$-functor $A \to
D^\varepsilon(A_\varepsilon)$ which “picks” (a variation of) the object $\Gamma$ shown in \cite{Nder} to compactly generate $D^\varepsilon(A_\varepsilon)$ together with $A$ (see \S \ref{leftadj}). Note that the shape of $\Gamma$ as a “two-sided cone” has a history going back to \cite{LowenCurvature, keller2009nonvanishing} (see also \cite{Blanc_Katzarkov_Pandit_2018, HaidenCY}) whereas the incarnation we use here and the relation with the filtered derived category is a main novelty from \cite{Nder}.

Finally, in the short appendix we observe that by its construction from a Hochschild 2-class, a categorical deformation $\T_\varepsilon$ of a category $\T$ is “as singular” as the base category, thus reproving the observation from \cite{Nder} that for a smooth dg algebra $A$, the 1-derived category of an (uncurved) dg deformation $A_\varepsilon$ of $A$ is a categorical resolution of the classical derived category $D(A_\varepsilon)$. This also shows a peculiar consequence of the theory: it follows immediately from the facts above that the derived category of a dg deformation is equivalent to a localization of the category obtained by gluing two copies of $D(A)$ along the cone of the Hochschild class induced by the deformation.

Let us highlight that the choice was made to work with the least amount possible of higher categorical machinery. Indeed, while to even define the Hochschild cohomology it is necessary to work with enhanced categories and not with their (triangulated) homotopy categories, we do little more than that: in particular, when talking about functors between higher categories, we will only consider them as objects of the \emph{homotopy category} of quasi-functors \cite{toen_derived_morita}, i.e. the derived category of bimodules or, equivalently, of the homotopy category of $A_\infty$-functors \cite{Ainffunct}. We will never make use of any enhancement of this category, instead only using its triangulated structure. This has the significant advantage of making the definitions simpler, and the proofs (arguably) more transparent. However, this approach leads to a theory that is less well-equipped to dealing with certain questions: for example, showing that $\HH^1(\T)$ parametrizes infinitesimal automorphism of a deformation necessarily requires using the full higher categorical structure of the functor category. We do not deal with these issues in the present paper; however, a more detailed discussion of the topic is present in \S \ref{coherentstuff}. Let us end this introduction by pointing out that the present results for first order deformations of pretriangulated dg categories constitute a first step in a program to interpret the full Hochschild complex, or rather its associated FMP, in terms of classification of (noncommutative) geometric objects. This would necessarily pass through generalizing the notion of categorical deformation to more general bases. This is work in progress.

\vspace{4mm}

\emph{Acknowledgements.} The first author is indebted to Dmitry Kaledin for a conversation at MFO, Oberwolfach in the spring of 2024 which helped shape some of the ideas in this paper. He also thanks Merlin Christ, Tobias Dyckerhoff, Julian Holstein, Bernhard Keller, Alexander Kuznetsov and Nicolò Sibilla for several helpful discussions.

The second author is deeply grateful to Dmitry Kaledin, Bernhard Keller and Michel Van den Bergh
for numerous illuminating conversations throughout the years on the topic of the curvature problem, and specifically for the collaborations on \cite{MoritaDef}, \cite{KaledinLowen} and \cite{LowenCurvature} 
which have inspired the present work.

\section{Preliminaries}
\subsection{Conventions}
All rings and algebras are associative and unital. Graded objects are assumed to be $\mathbb{Z}$-graded, and we employ cohomological grading. We'll denote with $k$ a fixed base field, and with $\ke$ the $k$-algebra $k[t]/(t^2)$. Unless otherwise specified, deformation will mean infinitesimal first order deformation, i.e. deformation over $\ke$. In keeping with the conventions from \cite{MoritacDef}, we will use the notation $A\Mod$ to denote the category of \emph{right} $A$-modules. 

To enhance our triangulated categories, we will use the model of dg categories; in fact, we will often not distinguish between a triangulated category and its enhancement. All dg categories are over the base field $k$. As a model for the category of functors between triangulated categories we'll usually employ the category of quasi-functors (with the exception of section \ref{leftadj} where we make use of $A_{\infty}$ functors). When talking about filtered derived categories, we'll denote with $D^\varepsilon(-)$ the 1-derived category that in \cite{Nder} is denoted with $D^1(-)$.  If $\T$ is a (pretriangulated) dg category, the complexes of morphisms in the category $\T$ will be denoted $\Hm{\T}(X,Y)$, and with $H^0(\T)$ we'll denote its homotopy category with its canonical triangulated structure. An exception to this convention will be given by derived categories of bimodules and quasi-functors: in that case, by $D(\op{\A}\otimes \B)$ we'll mean the \emph{homotopy category} of the relevant enhancement, and with $\Hm{D(\op\A\otimes \B)}(M, N)$ the $k$-module of morphisms in the derived category. To highlight this fact, we will at times also use the notation $\Ext^i_{\Fun(\A,\B)}(M,N)$ to signify $\Hm{D(\op\A\otimes \B)}(M, N[i])$. In the same vein, isomorphisms between quasi-functors will always mean isomorphisms in the homotopy category. 

\subsection{Operations on quasi-functors}
We refer the reader to \cite{kellerICM} for a general introduction to dg-categories and related constructions. Recall that if $\A, \B$ are dg categories, a \emph{quasi-functor} is an object $X\in D(\op{\A}\otimes\B)$ which is right quasi-representable, i.e. such that for every $A\in \A$, the element $X(-,A)$ is quasi-isomorphic to a representable $\B$-module; a quasi-functor is said to be a quasi-equivalence if it admits an inverse; this is equivalent to inducing an equivalence between the respective homotopy categories. 
\subsubsection{General operations}
Despite quasi-functors not being literal functors, we'll still use classical category-theoretical notations for various constructions. For simplicity, we will write $\A\tow{F}\B$ to denote a quasi-functor $F\in D(\op{\A}\otimes \B)$. Let $\A, \B, \C$ be dg categories. Given two quasi-functors $\A\tow{G}\B$ and $\B\tow{F}\C$, their composition $G\otimes^{\operatorname{L}}_\B F$ will be denoted with $FG$. Given $F,G\colon \B \to \C$, $ M, N \colon \A \to \B$, $f\in \Hm{D(\op{\A}\otimes \B)}(M, N)$ and $\eta\in \Hm{D(\op{\B}\otimes\, \C)}(F, G)$, we will denote with $\eta M$ the morphism $\id_M \otimes^{\operatorname{L}}\eta \in \Hm{D(\op{\A}\otimes\, \C)}(FM, GN)$ and with $F f$ the morphism $f\otimes^{\operatorname{L}} \id_F \in \Hm{D(\op{\A}\otimes \,\C)}(FM, GN)$. This is compatible with the classical use, as the following lemma shows.
\begin{Lem}
    The diagram 
\[\begin{tikzcd}
	FM & GM \\
	FN & GN
	\arrow["{\eta M}", from=1-1, to=1-2]
	\arrow["{F f }"', from=1-1, to=2-1]
	\arrow["{G f }", from=1-2, to=2-2]
	\arrow["{\eta N}", from=2-1, to=2-2]
\end{tikzcd}\]
commutes.
\end{Lem}
\begin{proof}
Writing down the definitions,
both compositions are seen to equal $f\otimes^{\operatorname{L}} \eta$.
\end{proof}

\subsubsection{Adjunctions}
We'll use extensively the notion of adjunction of quasi-functors, see \cite{genovese2015adjunctions}. Given two quasi-functors $\A\tow{F}\B$ and $\B \tow{G}\A$, we will say \cite[Definition 6.2]{genovese2015adjunctions} that $F$ is left adjoint to $G$ if there exist two morphisms $\id_\A\to FG$ in $D(\op{\A}\otimes \A)$ and $GF\to \id_\B$ in $D(\op{\B}\otimes \B)$ satisfying the usual unit-counit equation in the appropriate derived category. We will frequently need the following statement, which is straightforward to prove:
\begin{Prop}
    Let $\A, \B, \C$ be dg categories and $\A \tow{F}\B$ a quasi-functor with right adjoint $G$. Let $\C\tow{M}\A$ and $\C \tow{N}\B$. There is a natural isomorphism \[
    \Hm{D(\op{\C}\otimes \B)}(FM, N)\cong \Hm{D(\op{\C}\otimes \A)}(M, GN).
    \] 
\end{Prop}
Note that it follows from  \cite{dyckerhoffsphericalrelative} that this notion of adjunction is compatible with the one for the associated $\infty$-functors between stable $\infty$-categories.
\subsection{Recollements of quasi-functors}
Let $\A$, $\B$ be pretriangulated dg categories. A categorical extension of $\A$ by $\B$, or recollement, is a dg category $\C$ equipped with two quasi-functors \[
\A\tow{i}\C \tow{\I}\B
\] satisfying the following conditions:

\begin{enumerate}
    \item The composition $\I\circ i$ vanishes;
    \item The quasi-functor $i$ admits a left adjoint $Q$ and a right adjoint $K$; the quasi-functor $\I$ admits a left adjoint $G$. The adjunctions are induced by units \[
    \id_\C \tow{q} iQ \text{, } \id_\C \overunderset{g}{}{\to} \I G \text{, } \id_\A \tow{} Ki
    \] and counits \[
    Qi\tow{}\id_\A \text{, } iK\tow{k}\id_\C \text{, } G\I \tow{\xi} \id_\C.
    \]
    \item The transformations \[
    \id_\A\to Ki \text{, } Qi\to \id_\A \text{ and }\id_\C \tow{g} \I G
    \] are isomorphisms in the appropriate derived categories; this corresponds to the statement that $i$ and $G$ are homotopically fully faithful.
    \item The square 
\[\begin{tikzcd}
	G\I & {\id_\C} \\
	0 & iQ
	\arrow["\xi", from=1-1, to=1-2]
	\arrow[from=1-1, to=2-1]
	\arrow["q", from=1-2, to=2-2]
	\arrow[from=2-1, to=2-2]
\end{tikzcd}\] is homotopy cartesian in $D(\op\C\otimes \C)$, i.e. there exists a natural transformation $iQ\tow{\partial}G\I[1]$ fitting into a triangle \[
G\I \tow{\xi} \id_\C \tow{q} iQ \tow{\partial}G\I[1];
\] this corresponds to the classical fact that any object of $\C$ can be obtained as an extension of an object of $\A$ by an object of $\B$.
\end{enumerate}
If these conditions are satisfied, then $I$ also admits a right adjoint, which we will leave unlabeled. Pictorially, we draw  \[\begin{tikzcd}
	\A & {\C} & \B.
	\arrow["i", hook, from=1-1, to=1-2]
	\arrow["\I", from=1-2, to=1-3]
	\arrow["K", curve={height=-22pt}, from=1-2, to=1-1]
	\arrow["Q"', curve={height=22pt}, from=1-2, to=1-1]
	\arrow["G"', curve={height=22pt}, from=1-3, to=1-2]
	\arrow[ curve={height=-22pt}, from=1-3, to=1-2]
\end{tikzcd}\]

\begin{Rem}
    It is easy to show that a recollement of quasi-functors gives rise to a semiorthogonal decomposition \[
    H^0\C=\langle H^0i\A, H^0G\B \rangle
    \] of the homotopy category; conditions $(1)$ and $(2)$ imply that the two subcategories are indeed semiorthogonal, while condition $(4)$ implies that any object of $H^0\C$ can be obtained as the cone of a morphism from an object of $i\A$ to one of $G\B$.
\end{Rem}
\begin{Prop}
    There exists a natural transformation $K\tow{\alpha}Q$ fitting into a triangle \[
KG\I \tow{\gamma} K \tow{\alpha} Q \tow{\beta} KG\I[1].
\]
\end{Prop}
\begin{proof}
 
Apply the functor $K$ to the triangle 
\[
G\I \tow{\xi} \id_\C \tow{q} iQ \tow{\partial}G\I[1];
\]
and use that $KiQ\cong K$  .
\end{proof}
We will crucially use the natural transformation $\alpha$ in the definition of categorical deformation.
\subsection{Gluing dg categories}\label{gluing}We now recall some constructions which allow, given two dg categories and a quasi-functor between them, to construct an explicit recollement; for more details, see  \cite[Section 4.1]{Kuznetsov_2014}, or \cite{christ2024laxadditivity} for a treatment in terms of $\infty$-categories. Let $\A, \B$ be pretriangulated dg categories and $\phi$ a $\B$-$\A$ right quasi-representable bimodule. The category $\A\times_\phi \B$ has as objects the triples $(M_1, M_2, n)$ with $M_1\in \A$, $M_2\in \B$ and $n \in \phi(M_2, M_1)$ a closed degree $0$ element. The Hom-complexes are given by \[
\Hm{\A\times_\phi \B}((M_1, M_2, n), (N_1, N_2, n'))=\Hm{\A}(M_1, N_1)\oplus \Hm{\B
}(M_2, N_2) \oplus \phi[-1]
\] with differential \[
d(f_1, f_2, f_{21})=(df_1, df_2, -df_{21}+n'f_1-f_2n).
\] The category $\A\times_\phi \B$ is pretriangulated and comes equipped with several functors:
     \begin{itemize}
         \item Two embeddings \[\begin{split}
             i\colon \A &\to \A\times_\phi \B\\
             M &\to (M, 0,0)
         \end{split} 
         \] and \[\begin{split}
             G\colon \B &\to \A\times_\phi \B\\
             M &\to (0, M,0)
         \end{split} \]
         \item A left adjoint \[\begin{split}
             Q\colon  \A\times_\phi \B&\to \A\\
              (M_1, M_2, n) &\to M_1
         \end{split} 
         \] to $i$ and a right adjoint \[\begin{split}
             \I \colon\A\times_\phi \B &\to\B  \\
             (M_1, M_2, 0) &\to M_2
         \end{split} \] to $G$;
         \item A right adjoint to $i$ \[\begin{split}
         K\colon  \A\times_\phi \B&\to \A\Mod\\
         (M_1, M_2, n) &\to \operatorname{Cone}(n)[-1].
         \end{split}\] where we have used the Yoneda Lemma to see $n$ as a map $h_{M_1}\to \phi(M_2, -)$. 
     \end{itemize}
The functors $i, G, Q, \I$ are honest dg functors while in principle $K$ is only a bimodule; one can however show that it is always a quasi-functor.
     These functors satisfy the identities \begin{equation}\label{idgluing}
   Q i=Ki= \id_{\A}\text{, } \I G= \id_{\B}\text{, } QG=0 \text{, } KG=\phi[-1]\text{, } \I i=0.
         \end{equation}Note that the first two identities are induced by the (co)units of the respective adjunctions, expressing the fact the the (co)unit isomorphism is the identity. 
         
         \begin{Prop}
             The functors above fit into a recollement 
\[\begin{tikzcd}
	\A & {\A\times_\phi \B} & \B.
	\arrow["{i}", from=1-1, to=1-2]
	\arrow["{K}"', curve={height=28pt}, from=1-2, to=1-1]
	\arrow["{Q}", curve={height=-28pt}, from=1-2, to=1-1]
	\arrow["{\I}", from=1-2, to=1-3]
	\arrow["{G}"', curve={height=28pt}, from=1-3, to=1-2]
	\arrow[curve={height=-28pt}, from=1-3, to=1-2]
\end{tikzcd}\]
         \end{Prop}
\begin{proof}
As already observed, the composition $i\I$ is the zero functor and the functors $i$ and $G$ are fully faithful. Hence we are only left with showing the existence of the required triangle. For this, observe that \[G\I(M_1, M_2, n)=(0, M_2, 0) \text{, }iQ(M_1, M_2, n)=(M_1, 0, 0)\] and the natural transformations $G\I\tow{\xi} \id_\C$ and $\id_\C\tow{q} iQ$ are given by the obvious maps. There is a natural transformation $iQ\tow{\partial}G\I[1]$ given by \[
\partial_{(M_1, M_2, n)}\colon (M_1, 0, 0)\tow{(0,0, n)}(0, M_1[1], 0)
\] which we claim fits into a triangle \[
G\I \tow{\xi} \id_{\A\times_\phi \B} \tow{q} iQ \tow{\partial}G\I[1].
\] To see this, one uses the description of cones in $\A\times_\phi \B$ from \cite[Lemma 4.3]{Kuznetsov_2014}. Indeed, it is straightforward to check that the cone of $\partial[-1]$ is isomorphic the identity functor, and that the natural transformations $\xi$ and $q$ correspond to the canonical maps \[
G\I \to \operatorname{Cone}(\partial[-1])\to iQ.
\]
\end{proof}
Like in the general case, we also obtain a (canonical) triangle 
\[
K \tow{\alpha}Q \tow{\beta} KG\I[1] \tow{\gamma} K[1]
\] which, since $KG[1]=\phi$, is the same thing as a triangle \begin{equation}\label{alphagluing}
K \tow{\alpha} Q \tow{\beta} \phi \I \tow{\gamma} K[1].
\end{equation}

\subsection{Yoneda extensions in triangulated categories}
\label{Yoneda}
Let $\T$ be a triangulated category. In this section only, we will consider \emph{unenhanced} triangulated categories; this will be reflected in the fact that, given objects $A,B \in \T$, we will use the notation $\Ext^n(A,B)$ to denote the $k$-module of morphisms in $\T$ from $A$ to $B[n]$.

\begin{Def}
   An \emph{(Yoneda) $n$-extension} $\mathcal{E}$ from $A$ to $B$ consists of $n$ exact triangles 
   \begin{equation}
       C_i \rightarrow E_i \rightarrow C_{i+1} \rightarrow C_i[1]
   \end{equation}
   for $i = 0, \dots, n-1$ with $C_0 = B$ and $C_n = A$.
   The \emph{splicing sequence} of the given $n$-extension is the resulting sequence of morphisms
   \begin{equation}\label{splicing}
      \sigma(E) =( 0 \rightarrow B \rightarrow E_0 \rightarrow E_1 \rightarrow \dots \rightarrow E_{n-1} \rightarrow A \rightarrow 0).
   \end{equation}
   For another $n$-extension $\mathcal{E}' = (C_i' \rightarrow E_i' \rightarrow C_{i+1} \rightarrow)$, a \emph{morphism of $n$-extensions} $\mathcal{E} \rightarrow \mathcal{E}'$ consists of maps $g_i: C_i \rightarrow C_i'$ and $f_i: E_i \rightarrow E_i'$ with $g_0 = 1_B$, $g_n = 1_A$ inducing morphisms of all the relevant triangles.
\end{Def}

An $n$-extension $E = (C_i \rightarrow E_i \rightarrow C_{i+1} \rightarrow)$ determines $n$ connecting morphisms
\begin{equation}
    \varphi_i: C_{i+1} \rightarrow C_i[1].
\end{equation}
The \emph{Ext-class} of $\mathcal{E}$ is by definition the composition
\begin{equation}\label{simp}
    \varphi(\mathcal{E}) = \varphi_0[n-1]\varphi_1[n-2]\dots \varphi_{n-1}: A \rightarrow B[n].
\end{equation}
Let $\Ext^n_Y(A,B)$ denote the set of $n$-extensions from $A$ to $B$. Consider the map
\begin{equation}
    \varphi: \Ext^n_Y(A,B) \rightarrow \Ext^n(A,B): \mathcal{E} \mapsto \varphi(\mathcal{E}).
\end{equation}
The map $\varphi$ is readily seen to be surjective, as in fact any factorization of a morphism $A \rightarrow B[n]$ into an $n$-simplex as in \eqref{simp} allows for the construction of a corresponding pre-image.
\begin{Prop}\label{morphext}
    Consider $n$-extensions $\mathcal{E}$, $\mathcal{E}'$ from $A$ to $B$. If there exists a morphism $\mathcal{E} \rightarrow \mathcal{E}'$, then we have $\varphi(\mathcal{E}) = \varphi(\mathcal{E}')$.
\end{Prop}

\begin{proof}
    This easily follows from the definition of $\varphi$ and the requirement that a morphism fixes both $A$ and $B$.
\end{proof}

\begin{Rem}
    Let $\Sigma^n(A,B)$ denote the set of sequences of shape \eqref{splicing} in which the composition of every two consecutive maps is zero. Unlike in the familiar case of extensions in an abelian category, the map $\sigma: \Ext^n_Y(A,B) \rightarrow \Sigma^n(A,B)$ is in general neither surjective nor injective, due to the fact that cones are merely weak cokernels and the map towards the cone fails to be an epimorphism in general.
\end{Rem}

\section{Categorical deformations}\label{catdefinitions}
In this section we introduce the notion of categorical deformation of a triangulated category (Definition \ref{defdef}) and show how to any deformation corresponds an Hochschild class.
\subsection{Definition and basic properties}\label{defsec}
\begin{Def}\label{defdef}
    Let $\T$ be a pretriangulated dg category. A \emph{first order categorical deformation} of $\T$ consists of the following data:
    \begin{enumerate}
        \item A recollement $\T_{\varepsilon}$ as in \[\begin{tikzcd}
	\T & {\T_\varepsilon} & \T
	\arrow["i", hook, from=1-1, to=1-2]
	\arrow["\I", from=1-2, to=1-3]
	\arrow["K", curve={height=-22pt}, from=1-2, to=1-1]
	\arrow["Q"', curve={height=22pt}, from=1-2, to=1-1]
	\arrow["G"', curve={height=22pt}, from=1-3, to=1-2]
	\arrow[curve={height=-22pt}, from=1-3, to=1-2]
\end{tikzcd}\] with associated $\alpha\colon K \rightarrow Q$;
        \item A Yoneda $2$-extension from $I$ to $I$ compatible with the semiorthogonal decomposition, i.e. whose splicing sequence is of the form
        \begin{equation}
            0 \rightarrow \I \tow{\delta_1} K \tow{\alpha} Q \tow{\delta_2} \I \to 0
        \end{equation}
        for some $\delta_1, \delta_2$.
    \end{enumerate}
\end{Def}
In other words, the deformation is given by a recollement together with two triangles  \[
\begin{split}
    \I \tow{\delta_1} K \tow{\zeta_1} C \tow{\eta_1} \I[1]\text{ and } C \tow{\eta_2} Q \tow{\delta_2} \I \tow{\zeta_2} C[1]
\end{split}
\] with the property that the composition \[
K \tow{\zeta_1} C \tow{\eta_2}Q
\] coincides with $\alpha$ as a map in the derived category. 
We will routinely abuse notation and denote with $\T_\varepsilon$ the pair of recollement and extension. 

From this data, we easily get a Hochschild class: we have a map $\iota\colon \I[-1] \to \I[1]$ defined as the composition of the two boundary maps $
\I[-1]\tow{\zeta_2[-1]}C\tow{\eta_1}\I[1]$; 
and we can compose $\iota$ on the right with $G$ to obtain a natural transformation $\I G[-1]\tow{\iota G} \I G[1]$; recalling now the natural isomorphism $\id_\T\tow{g} \I G$, we obtain the class \[
\mu_\T(\T_\varepsilon)\in \HH^2(\T)=\Ext^2_{\Fun(\T,\T)}(\id_\T, \id_\T)
\]
as the unique morphism $\id_\T[-1]\to \id_\T[1]$ in $D(\op{\T}\otimes \T)$ fitting in the commutative diagram
\[\begin{tikzcd}
	{\I G[-1]} & {\I G[1]} \\
	{\id_{\T}[-1]} & {\id_{\T}[1].}
	\arrow["{\iota G}", from=1-1, to=1-2]
	\arrow["g", from=2-1, to=1-1]
	\arrow["{\mu_\T(\T_\varepsilon)}", from=2-1, to=2-2]
	\arrow["g"', from=2-2, to=1-2]
\end{tikzcd}\] In the language of section \ref{Yoneda}, the class $\mu_\T(\T_\varepsilon)$ is obtained by taking the Ext-class associated to the extension with splicing sequence
\[
\id_\T \tow{\delta_1Gg}KG \tow{\alpha G} QG \tow{g^{-1}\delta_2 G} \id_\T.
\]

\begin{Lem}\label{triaglue}
If $\T_\varepsilon$ is a categorical deformation of a triangulated category $\T$, then there is a triangle 
\begin{equation}\label{octa1}
\I[-1]\tow{\iota} \I[1] \tow{\sigma} KG\I[1] \tow{\pi} \I \end{equation}in $D(\op{\T_\varepsilon}\otimes \T)$. Moreover, consider the diagram
\[\begin{tikzcd}
	& {KG\I[1]} \\
	{\I[1]} && \I \\
	K && Q \\
	& C
	\arrow["\pi", from=1-2, to=2-3]
	\arrow[dashed, from=1-2, to=3-1]
	\arrow["\sigma", from=2-1, to=1-2]
	\arrow["{\delta_1}"', dashed, from=2-1, to=3-1]
	\arrow["\iota"', dashed, from=2-3, to=2-1]
	\arrow["{\zeta_2}"'{pos=0.4}, dashed, from=2-3, to=4-2]
	\arrow["\alpha", from=3-1, to=3-3]
	\arrow["{\zeta_1}"', from=3-1, to=4-2]
	\arrow[from=3-3, to=1-2]
	\arrow["{\delta_2}"', from=3-3, to=2-3]
	\arrow["{\eta_1}"'{pos=0.6}, from=4-2, to=2-1]
	\arrow["{\eta_2}"', from=4-2, to=3-3]
\end{tikzcd}\] where the dotted arrows are of degree $1$. Then $\sigma$ and $\pi$ form commutative triangles with the faces, other than \eqref{octa1}, that contain them.
\end{Lem}
\begin{proof}
    This is the octahedral axiom for triangulated categories.
\end{proof}
Composing with $G$ and using $\I G\cong \id_\T$, we obtain:
\begin{Cor}\label{triacor}
There is a triangle
\begin{equation}\label{triacone}
    \id_\T[-1] \tow{\mu_\T(\T_\varepsilon)} \id_\T[1]\to KG[1] \to \id_\T
\end{equation} in $D(\op{\T}\otimes \T)$.
\end{Cor}
\subsection{Equivalences of deformations}
\label{equisec}
Let $\T_\varepsilon, \Tilde{\T}_\varepsilon$ be deformations of a triangulated category $\T$. Consider a quasi-functor \[
F\colon\T_\varepsilon \to \Tilde{\T}_\varepsilon.
\] We'll say that $F$ is compatible with the semiorthogonal decompositions if the diagram 
\[\begin{tikzcd}
	\T & {\T_\varepsilon} & \T \\
	\T & {\Tilde{\T_\varepsilon}} & \T
	\arrow["{\id_\T}"', from=1-1, to=2-1]
	\arrow["Q"', from=1-2, to=1-1]
	\arrow["\I", from=1-2, to=1-3]
	\arrow["F"', from=1-2, to=2-2]
	\arrow["{\id_\T}"', from=1-3, to=2-3]
	\arrow["{\Tilde{Q}}"', from=2-2, to=2-1]
	\arrow["{\Tilde{\I}}", from=2-2, to=2-3]
\end{tikzcd}\] is $2$-commutative, i.e. if there exist two natural isomorphisms \begin{equation}\label{sodmap}
\chi_{\I}\colon \I \to \Tilde{\I}F\text{ and }\chi_Q\colon Q\to \Tilde{Q}F.
\end{equation} \begin{Prop}
    The following are equivalent:
    \begin{itemize}
        \item The functor $F$ is compatible with the semiorthogonal decompositions;
        \item There exists a triangle
        \begin{equation}\label{trianglemap}
\Tilde{G}\I \to F \to \Tilde{i}Q \to \Tilde{G}\I [1].
    \end{equation}
in $D(\op{\T_\varepsilon}\otimes \Tilde{\T}_\varepsilon).$
    \end{itemize}
\end{Prop}
\begin{proof}
    Assume that isomorphisms as in \eqref{sodmap} are given. Since $\Tilde{\T_\varepsilon}$ is a categorical extension, there exists a triangle \[
    \Tilde{G}\Tilde{\I}\to \id_{\Tilde{\T_\varepsilon}}\to \Tilde{i}\Tilde{Q} \to  \Tilde{G}\Tilde{\I}[1];
    \] composing on the right with $F$ yields a triangle
    \[
    \Tilde{G}\Tilde{\I}F\to F\to \Tilde{i}\Tilde{Q}F \to  \Tilde{G}\Tilde{\I}F[1];
    \]under the isomorphisms \eqref{sodmap}, this yields the triangle \eqref{trianglemap}. Vice versa, assume given a triangle as in \eqref{trianglemap}. Composing $\Tilde{G}\I \to F$ on the left with $\Tilde{\I}$ and using that $\Tilde{\I}\Tilde{G}\cong \id_{\T},$ we obtain a map $\I \to\Tilde{\I}F$ whose cone is given by $\Tilde{\I}\Tilde{i}Q$; since $\Tilde{\I}\Tilde{i}\cong 0$, this is an isomorphism. In the same way, one can compose on the left with $\Tilde{Q}$ to obtain a natural isomorphism $Q\to \Tilde{Q}F$.
    
\end{proof}

Using the existence of the triangle \eqref{trianglemap} and reasoning as in the proof above, one also proves the following:
\begin{Cor}
    If $F$ is compatible with the semiorthogonal decompositions, then there exist further natural isomorphisms \[
    \chi_G\colon \Tilde{G}\to FG \text{ and } \chi_i\colon \Tilde{i}\to Fi.
    \]
\end{Cor}

We'll need the following technical fact:
\begin{Lem}\label{diachasing}
    The diagram
\[\begin{tikzcd}
	& \I G \\
	{\id_\T} & {\Tilde{\I}FG} \\
	& {\Tilde{\I}\Tilde{G}}
	\arrow["{\chi_{\I} G}", from=1-2, to=2-2]
	\arrow["g", from=2-1, to=1-2]
	\arrow["{\Tilde{g}}"', from=2-1, to=3-2]
	\arrow["{\Tilde{\I}\chi_G }"', from=3-2, to=2-2]
\end{tikzcd}\]
commutes.
\end{Lem}
\begin{proof}
    Follows from writing down the various definitions.
\end{proof}

Recall that the deformation $\T_\varepsilon$ comes equipped with a class $\iota\in \Ext^2_{\Fun(\T_\varepsilon,  \T)}(\I, \I)$; similarly, $\Tilde{\T_\varepsilon}$ comes with a class 
$\Tilde{\iota}\in \Ext^2_{\Fun(\Tilde{\T_\varepsilon}, \T)}(\Tilde{\I}, \Tilde{\I}).$ Composing on the right with $F$ determines a natural map \[
\Ext^2_{\Fun(\Tilde{\T_\varepsilon}, \T)}(\Tilde{\I}, \Tilde{\I})\to \Ext^2_{\Fun({\T_\varepsilon}, \T)}(\Tilde{\I}F, \Tilde{\I}F)
\] while the natural isomorphism $\chi_\I$ gives an isomorphism \[
\Ext^2_{\Fun({T_\varepsilon}, \T)}(\Tilde{\I}F, \Tilde{\I}F)\cong  \Ext^2_{\Fun({T_\varepsilon}, \T)}(I,I).
\] We'll say that a quasi-equivalence $F$ (compatible with the SODs) is an \emph{equivalence of deformations} if the class $\Tilde{\iota}$ corresponds to $\iota$ under these maps, i.e. if the diagram 
\begin{equation}\label{eqdefs}
\begin{tikzcd}
	{\I[-1]} & {\I[1]} \\
	{\Tilde{\I}F[-1]} & {\Tilde{\I}F[1]}
	\arrow["\iota", from=1-1, to=1-2]
	\arrow["{\chi_{\I}[-1]}"', from=1-1, to=2-1]
	\arrow["{\chi_{\I}[1]}", from=1-2, to=2-2]
	\arrow["{\Tilde{\iota}F}", from=2-1, to=2-2]
\end{tikzcd}    
\end{equation}
commutes.
\begin{Rem}
    It is not strictly necessary to assume $F$ to be a quasi-equivalence; however, it is always true that if a functor exists that is compatible with the SODs and that preserves the Ext-class, then there exists a quasi-equivalence with the same properties.
\end{Rem}
\begin{Prop}\label{equivalenceofclass}
If there exists an equivalence between two deformations $\T_\varepsilon$ and $\Tilde{\T_\varepsilon}$, then $\mu_\T(\T_\varepsilon)=\mu_\T(\Tilde{\T_\varepsilon})$
\end{Prop}
\begin{proof}
    Let $F$ be such an equivalence. By definition, the diagram (\ref{eqdefs})
commutes, hence the diagram 
\[\begin{tikzcd}
	{\I G[-1]} & {\I G[1]} \\
	{\Tilde{\I}FG[-1]} & {\Tilde{\I}FG[1]}
	\arrow["{\iota G}", from=1-1, to=1-2]
	\arrow["{\chi_{\I} G[-1]}"', from=1-1, to=2-1]
	\arrow["{\chi_{\I} G[1]}", from=1-2, to=2-2]
	\arrow["{\Tilde{\iota}FG}", from=2-1, to=2-2]
\end{tikzcd}\] commutes as well. On the other hand, the diagram
\[\begin{tikzcd}
	{\Tilde{\I}\Tilde{G}[-1]} & {\Tilde{\I}\Tilde{G}[1]} \\
	{\Tilde{\I} F G[-1]} & {\Tilde{\I}FG[1]}
	\arrow["{\Tilde{\iota}\Tilde{G}}", from=1-1, to=1-2]
	\arrow["{\Tilde{\I}\chi_G[-1]}"', from=1-1, to=2-1]
	\arrow["{\Tilde{\I}\chi_G[1]}", from=1-2, to=2-2]
	\arrow["{\Tilde{\iota}FG}", from=2-1, to=2-2]
\end{tikzcd}\]
commutes by naturality of $\Tilde{\iota}$. Pasting these and using Lemma \ref{diachasing}, we know that the diagram
\[\begin{tikzcd}
	& {\I  G[-1]} & {\I  G[1]} \\
	{\id_\T[-1]} & {\Tilde{\I} F G[-1]} & {\Tilde{\I}FG[1]} & {\id_\T[1]} \\
	& {\Tilde{\I}\Tilde{G}[-1]} & {\Tilde{\I}\Tilde{G}[1]}
	\arrow["{\iota G}", from=1-2, to=1-3]
	\arrow["{\chi_{\I}  G[-1]}", from=1-2, to=2-2]
	\arrow["{\chi_{\I}  G[1]}"', from=1-3, to=2-3]
	\arrow["g", from=2-1, to=1-2]
	\arrow[from=2-1, to=2-2]
	\arrow["{\Tilde{g}}"', from=2-1, to=3-2]
	\arrow["{\Tilde{\iota}FG}", from=2-2, to=2-3]
	\arrow["g"', from=2-4, to=1-3]
	\arrow[from=2-4, to=2-3]
	\arrow["{\Tilde{g}}", from=2-4, to=3-3]
	\arrow["{\Tilde{\I}\chi_G[-1]}"', from=3-2, to=2-2]
	\arrow["{\Tilde{\iota}\Tilde{G}}", from=3-2, to=3-3]
	\arrow["{\Tilde{\I}\chi_G[1]}", from=3-3, to=2-3]
\end{tikzcd}\] commutes, i.e. $\mu_\T(\T_\varepsilon)=\mu_\T(\Tilde{\T_\varepsilon)}$.
\end{proof}
\begin{Def}
    Define the set $\CatDef_\T({\ke})$ as the set of categorical deformations of $\T$ up to equivalence of deformations.
\end{Def}

By Proposition \ref{equivalenceofclass}, the map that assigns to a categorical deformation $\T_\varepsilon$ the class $\mu_\T(\T_\varepsilon)$ descends to the quotient, defining a morphism \[
\mu_\T\colon \CatDef_\T(\ke)\to \HH^2(\T).
\]

Our first main result is the following

\begin{Thm}\label{HH2Defo}
The map $\mu_\T$ defines a bijection between $\HH^2(\T)$ and the set $\CatDef_\T(\ke)$ of equivalence classes of categorical deformations of $\T$.
\end{Thm}
To prove Theorem \ref{HH2Defo}, we construct an explicit inverse to $\mu_\T$ in \S \ref{inverse}. The proof will be completed in \S \ref{proofof}.

\subsection{Constructing an inverse}\label{inverse}

\subsubsection{The construction}
Consider a class $\mu\in \HH^2(\T)$; we see $\mu$ as a natural transformation $\id_\T[-1]\tow{\mu} \id_\T[1]$, and complete it (non-canonically) to a triangle \[
\id_\T[-1]\tow{\mu} \id_\T[1] \tow{\nu}\phi \tow{\omega}\id_\T
\] in $D(\op{\T}\otimes \T)$. By construction $\phi$ is a quasi-functor $\T\to \T$ and we can construct the gluing \[
\T_\varepsilon= \T \times_\phi \T;
\] this comes equipped with the various functors and natural transformations described in \ref{gluing}. To construct the Yoneda extension, begin by defining the natural transformation 
$\I \tow{\delta_1} K$ as the composition \[
\I \tow{\nu \I} \phi \I[-1] \tow{\gamma} K
\] where $\gamma$ comes from \eqref{alphagluing}. Complete now $\delta_1$ to a triangle \[
\I \tow{\delta_1} K\tow{\zeta_1} C \tow{\eta_1} \I[1].
\] The octahedral axiom proves:
\begin{Lem}\label{oct2}
There is a triangle \begin{equation}\label{octa2}
C \tow{\eta_2}Q \tow{\delta_2}\I  \tow{\zeta_2} C[1].
\end{equation} in $D(\op{\T_\varepsilon}\otimes \T)$. Moreover, consider the diagram
\[\begin{tikzcd}
	& C \\
	{\I[-1]} && Q \\
	\I && K \\
	& {\phi \I [-1]}
	\arrow["{\eta_2}", from=1-2, to=2-3]
	\arrow["{\eta_1}"{pos=0.6}, dashed, from=1-2, to=3-1]
	\arrow["{\zeta_2}", from=2-1, to=1-2]
	\arrow["{\mu \I}"', dashed, from=2-1, to=3-1]
	\arrow["{\delta_2}"', dashed, from=2-3, to=2-1]
	\arrow[dashed, from=2-3, to=4-2]
	\arrow["{\delta_1}", from=3-1, to=3-3]
	\arrow["{\nu \I}"', from=3-1, to=4-2]
	\arrow["{\zeta_1}"{pos=0.4}, from=3-3, to=1-2]
	\arrow["\alpha"', from=3-3, to=2-3]
	\arrow[from=4-2, to=2-1]
	\arrow["\gamma"', from=4-2, to=3-3]
\end{tikzcd}\]
 where the dotted arrows have degree $1$. Then $\zeta_2$ and $\eta_2$ form commutative triangles with the faces other than \eqref{octa2} that contain them.
\end{Lem}

Hence, the composition $
K \tow{\zeta_1}C \tow{\eta_2} Q $ equals $\alpha$ and the two triangles \[
\I \tow{\delta_1}K \tow{\zeta_1} C \tow{\eta_1} \I[1] \text{ and } C \tow{\eta_2}Q \tow{\delta_2} \I \tow{\zeta_2} C[1]
\] give a Yoneda 2-extension compatible with the semiorthogonal decomposition. We have thus proved:
\begin{Prop}
For any Hochschild class $\mu\in \HH^2(\T)$, the category $\T_\varepsilon$ constructed above is a categorical deformation of $\T$.
\end{Prop}
\subsubsection{Well-definedness}\label{welldef}
In our construction we made two arbitrary choices: the choices of the cones $\phi$ of $\mu$ and $C$ of $\delta_1$; we'll show that in each casedifferent choices give rise to equivalent deformations. For $\delta_1$, this is very easy; since different choices of cones do not alter the underlying category, we can pick the identity functor as an equivalence; this clearly preserves the class $\iota$, since that is not altered by a different choice of cone for $\delta_1$. On the other hand, changing the choice of cone of $\mu$ does change the underlying category; let $\phi'$ be another choice for that cone. Then we always have a (noncanonical) isomorphism $f\colon \phi\to \phi'$ in $D(\op\T\otimes \T)$. Assume for now that this $f$ is given by a honest map of dg bimodules; then one can define the functor \[
\begin{split}
    \T_\varepsilon \times_\phi \T_\varepsilon &\to \T_\varepsilon \times_{\phi'}\T_\varepsilon\\
    (M_1, M_2, n) &\to (M_1, M_2, f(n))
\end{split}
\] which, for the same reason as above, is readily seen to be an equivalence of deformations. In the case where $f$ is instead given by a zigzag of quasi-isomorphism, we simply obtain a zigzag of quasi-functors.

\section{Proof of Theorem \ref{HH2Defo}}\label{proofof}
In the previous section, we have shown that there is a well-defined map \[
\HH^2(\T)\to \CatDef_\T(\ke)
\] which assigns to a Hochschild class the gluing along its cone. In this section we show that this gives an inverse of $\mu_\T$.
\subsection{Class to class}\label{class2class}
First we have to show that if we start with a class $\mu$, take its cone $\phi$, construct the gluing $\T_\varepsilon=\T\times_\phi \T$ and then take the class $\mu_\T(\T_\varepsilon)$ we recover the starting class $\mu$. This essentially follows from Lemma \ref{oct2}: by construction, the class $\mu_\T(\T_\varepsilon)$ is obtained by first taking the composition \[
\I[-1] \tow{\zeta_2}C \tow{\eta_1} \I[1]
\] and then composing it with $G$. By Lemma \ref{oct2} we know that $\eta_1 \zeta_2=\mu \I$, so we only have to show that the diagram
\[\begin{tikzcd}
	{\id_{\T}[-1]} & {\id_{\T}[1]} \\
	{\I  G[-1]} & {\I  G[1]}
	\arrow["\mu", from=1-1, to=1-2]
	\arrow["g", from=1-1, to=2-1]
	\arrow["g"', from=1-2, to=2-2]
	\arrow["{\mu \I G}", from=2-1, to=2-2]
\end{tikzcd}\] commutes; but that follows by naturality of $\mu$.
\subsection{Deformation to deformation}
We now have to prove that if we start with a deformation $\T_\varepsilon$, take the associated class $\mu_\T(\T_\varepsilon)$, take its cone $\phi$, then the gluing $\T_\varepsilon'=\T \times_\phi \T$ is equivalent to $\T_\varepsilon$ as a deformation.

We first ought to construct a suitable quasi-functor $F\colon \T_\varepsilon\to\T\times_\phi\T $; we have the triangle  \[
    \id_\T[-1]\tow{\mu_\T(\T_\varepsilon)} \id_\T[1]\tow{}KG[1]\to \id_\T;
    \] from Corollary \ref{triacor} and, since choosing a different cone for $\mu_\T(\T_\varepsilon)$ yields equivalent deformations (see Section \ref{welldef}) we can pick $\phi=KG[1]$. Denoting with $\Tilde{i}, \Tilde{G}, \Tilde{K}$ the relevant functors of the gluing $\T\times_{KG[1]}\T$, observe that $\Tilde{K}\Tilde{i}=KG$ by construction -- the composition of the two adjoints always equals the gluing functor. Recall that, by virtue of being an extension, the category $\T_\varepsilon$ comes equipped with a morphism \[
iQ\to G\I[1]
\] in $D(\op{\T_\varepsilon}\otimes \T_\varepsilon)$. We then have \[
\begin{split}
    \Hm{D(\op{\T_\varepsilon}\otimes \T_\varepsilon)}(iQ, G\I[1])&\cong \Hm{D(\op{\T_\varepsilon}\otimes \T)}(Q, KG\I[1])= \Hm{D(\op{\T_\varepsilon}\otimes \T)}(Q, \Tilde{K}\Tilde{G}\I [1])\\ &\cong \Hm{D(\op{\T_\varepsilon}\otimes (\T\times_{KG}\T))}(\Tilde{i}Q, \Tilde{G}\I [1])
\end{split}
\] whence we obtain a morphism $\Tilde{i}Q\to \Tilde{G}\I [1]$ that we can complete to a triangle \[
\Tilde{G}\I \to F \to \Tilde{i}Q\to \Tilde{G}\I [1].
\] By the same argument as \cite[Proposition 7.7]{Kuznetsov_2014} the bimodule $F$ is a quasi-functor. To see that $F$ is a quasi-equivalence, the same process can be applied to obtain a triangle \[
G\Tilde{\I}\to F' \to i\Tilde{Q}\to G\Tilde{\I}[1]
\]  and one sees that $F'$ is a quasi-inverse to $F$.
\begin{Rem}
    Heuristically, the functor $F$ is  given by \[
    \begin{split}
    \T_\varepsilon&\to \T\times_{KG[1]}\T\\
    M &\to (QM, \I M, QM \to KG\I M[1]),
    \end{split}
    \]where the map $QM\to KG\I M[1]$ is obtained from the canonical map $iQ\to G\I[1]$ via the isomorphism \[
    \Hm{}(iQ, G\I[1])\cong \Hm{}(Q, KG\I[1]).
    \] One can then also directly show along the lines of \cite[Proposition 4.14]{Kuznetsov_2014} that $F$ is an equivalence.
\end{Rem}

To conclude, we'll need the following lemma.
\begin{Lem}\label{recovermap}
    Let $\T_\varepsilon$ be a categorical deformation. Then the induced natural transformation $\iota\colon \I[-1] \to \I[1]$ coincides with $\mu_\T(\T_\varepsilon) \I$.
\end{Lem}
\begin{proof}
    Consider the natural isomorphism $\I \tow{g\I}\I G\I$. The diagram
\[\begin{tikzcd}
	{\I[-1]} & {\I G\I[-1]} \\
	{\I[1]} & {\I G\I[1]}
	\arrow["g\I", from=1-1, to=1-2]
	\arrow["\iota", from=1-1, to=2-1]
	\arrow["\iota G\I"', from=1-2, to=2-2]
	\arrow["g\I", from=2-1, to=2-2]
\end{tikzcd}\]
commutes by naturality of $\iota$. The natural transformation $\mu_\T(\T_\varepsilon)$ is defined via the square
\[\begin{tikzcd}
	{\id_\T[-1]} & {\I G[-1]} \\
	{\id_\T[1]} & {\I G[1]}
	\arrow["g", from=1-1, to=1-2]
	\arrow["{\mu_\T(\T_\varepsilon)}", from=1-1, to=2-1]
	\arrow["\iota G"', from=1-2, to=2-2]
	\arrow["g", from=2-1, to=2-2]
\end{tikzcd}\] hence the diagram
\[\begin{tikzcd}
	{\I[-1]} & {\I G\I[-1]} \\
	{\I[1]} & {\I G\I[1]}
	\arrow["{g \I}", from=1-1, to=1-2]
	\arrow["{\mu_\T(\T_\varepsilon) \I}", from=1-1, to=2-1]
	\arrow["{\iota G\I}", from=1-2, to=2-2]
	\arrow["{g \I}", from=2-1, to=2-2]
\end{tikzcd}\] must commute. Pasting these we obtain a commutative diagram
\[\begin{tikzcd}
	{\I[-1]} & {\I G\I[-1]} & {\I[-1]} \\
	{\I[1]} & {\I G\I[1]} & {\I[1]}
	\arrow["{g \I}", from=1-1, to=1-2]
	\arrow["{\mu_\T(\T_\varepsilon) \I.}", from=1-1, to=2-1]
	\arrow["{\iota G\I}", from=1-2, to=2-2]
	\arrow["g\I"', from=1-3, to=1-2]
	\arrow["\iota", from=1-3, to=2-3]
	\arrow["{g \I}", from=2-1, to=2-2]
	\arrow["g\I"', from=2-3, to=2-2]
\end{tikzcd}\]
Since the horizontal arrows are isomorphisms, we get the claim.
\end{proof}
\begin{Rem}
    For the proof we never used that $\T_\varepsilon$ was a deformation, but only that it is an extension. Indeed the same proof shows that any natural transformation $\I \to \I$ can be recovered from the induced map $\id_\T\cong \I G\to \I G\cong \id_\T$.
\end{Rem}
We can now prove that $F$ is an equivalence of deformations. The goal is to show that the diagram 
\[\begin{tikzcd}
	{\I F[-1]} & {\I F[1]} \\
	{\I'[-1]} & {\I'[1]}
	\arrow["{\iota F}", from=1-1, to=1-2]
	\arrow["{\chi_{\I}}", from=1-1, to=2-1]
	\arrow["{\chi_{\I}}", from=1-2, to=2-2]
	\arrow["{\iota'}", from=2-1, to=2-2]
\end{tikzcd}\] is commutative. Using Lemma \ref{recovermap} and that by definition $\iota'=\mu_\T(\T_\varepsilon)\I'$, this is the same diagram as 
\[\begin{tikzcd}[column sep=2.5em]
	{\I F[-1]} & {\I F[1]} \\
	{\I'[-1]} & {\I'[1]}
	\arrow["{\mu(\T_\varepsilon)\I F}", from=1-1, to=1-2]
	\arrow["{\chi_\I}", from=1-1, to=2-1]
	\arrow["{\chi_\I}", from=1-2, to=2-2]
	\arrow["{\mu(\T_\varepsilon)\I'}", from=2-1, to=2-2]
\end{tikzcd}\] and now the claim follows by naturality of $\mu_\T(T_\varepsilon)$.

\section{Extending $A_\infty$-functors}\label{secextend}
In this section we record some technical results about extensions of $A_\infty$-functors that will be needed in the later section (Proposition \ref{propextend}); we also observe a relation between this construction and Hochschild cohomology.
\subsection{On enhancements}
Let $A$ be a dg algebra. We will work with a specific model for the (enhanced) derived category $D(A)$, given by the category $\operatorname{Tw}(A)$ of (one-sided) \emph{twisted complexes}. An object of this category is given by an ordinal $I$ and a pair \[
M=(\oplus_{i\in I} A[n_i], \{f_{ij}\}_{i,j\in I})
\] where $f_{ij}\in A[n_i-n_j]$ have the property that $f_{ij}=0$ for $i\leq j$ and $df+f^2=0$ (see e.g. \cite{BLL_pretriangulated, lowchar} for more details). The Hom-complex between two twisted object is given by the appropriate space of matrices, with twisted differential. The object $M$ can be seen as the graded $A$-module $\oplus_i A[n_i]$ with differential given by $d+f$.  There exists a fully faithful totalization dg functor \[
\operatorname{Tw}(A)\to A\Mod
\] which sends a twisted module to the just described $A$-module. Its essential image is given by the semi-free $A$-modules, hence $\operatorname{Tw}(A)$ gives a dg enhancement of $D(A)$. In the following, we will not distinguish between $D(A)$ and its enhancement $\operatorname{Tw}(A)$.
\subsection{Extension of functors}
 If $\A$ and $\B$ are pretriangulated dg categories, denote with $\Fun_{A_\infty}(\A, \B)$ the dg category of $A_\infty$-functors from $\A$ to $\B$ \cite{Ainffunct}; similarly, if $\A$ and $\B$ admit small coproducts, denote with $\Fun_{A_\infty}^c(\A, \B)$ the dg category of cocontinuous $A_\infty$-functors from $\A$ to $\B$ - that is, $A_\infty$-functors for which the underlying $H^0(\A) \rightarrow H^0(\B)$ preserves small coproducts. It is a well-known fact (see \cite[Proposition 3.27]{WellGenRamos}) that the restriction dg functor \[
\Fun_{A_\infty}^c(D(A), \B)\to\Fun_{A_\infty}(A, \B) 
\] along the Yoneda embedding $A\to D(A)$ is a quasi-equivalence. The enhancement $\operatorname{Tw}(A)$ allows for an easy description of a quasi-inverse to the restriction, given by the natural extension of an $A_\infty$-functor to the category of twisted objects. This is well-known and often used implicitly in the literature (see e.g. \cite[Section 2]{HaidenCY}); similar computations also appear in \cite{AnnoSph, AnnoBarDer}. For future use, we describe explicitly the formulas appearing in the extension:
\begin{Prop}\label{propextend}
    Any $A_\infty$-functor $
    F\colon A \to \B$ can be canonically extended to a cocontinuous $A_\infty$-functor \[
    \chi_A(F)\colon D(A)\to \B.
    \] This assignment provides a dg functor \[
    \chi_A\colon \Fun_{A_\infty}(A, \B)\to \Fun_{A_\infty}(D(A), \B)
    \] which preserves triangles in the homotopy category.
\end{Prop}
\begin{proof}[Construction of the extension]
  Preliminarily, observe that since $\B$ is pretriangulated and closed under coproducts, there is a fully faithful totalization dg functor \[
  \operatorname{Tw}(\B)\to \B.
  \] Hence, it will be enough to describe $\chi_A(F)$ as a functor $\operatorname{Tw}(A)\to \operatorname{Tw}(\B)$. For simplicity we omit the signs from the formulas; the interested reader can consult \cite{lowchar} for the precise sign conventions. By definition, $f$ is given by an object $FA\in \B$ together with an $A_\infty$-algebra morphism (see Section \ref{leftadj}) \[
  A\to \Hm{\B}(FA,FA)
  \] given by components $F_i\colon A^{\otimes i}\to \Hm{\B}(FA, FA)$. Given $M=(\oplus A_i[n_i], \delta_M)$ in $\operatorname{Tw}(A)$ set  \[
    F(\delta_M)=\sum_i F_i(\delta_M^{\otimes i})
    \] and define 
    $\chi_A(F)(M)=(\oplus_i fA[n_i], F(\delta))$.
    Given \[M=(\oplus_{i\in I}A[n_i], \delta_M)\text{, } N=(\oplus_{i\in J}A[n_j], \delta_N) \text{ and }
    g\in \Hm{\operatorname{Tw}(A)}(M, N),
    \]define then \[ \chi_A(F)_1(g)=\sum_{i\geq 1}F_i(\delta^k\otimes g\otimes  \delta^{\otimes l});\]
  in general, given $M_0, \ldots M_n\in \operatorname{Tw}(A)$ and $g^i\in \Hm{\operatorname{Tw}(A)}(M_{i-1}, M_{i})$ define the component $
    \chi_A(F)_n(g^1, \ldots g^n)$ as \[ \sum_{i\geq n} \sum_{k_0+\ldots k_n=i-n}f_i(\delta^{\otimes k_0}\otimes g^1\otimes\delta^{\otimes k_1}\otimes g^2\otimes\ldots \otimes \delta^{\otimes k_{n-1}} \otimes g^n \otimes \delta^{\otimes k_n}).
    \] One verifies that these formulas define an $A_\infty$-functor. To enhance $\chi_A(-)$ to a dg functor, one must define an action of natural transformation, that is: given two functors \[
    f, g \colon A \to \B
    \] and an $A_\infty$ natural transformation $\eta\colon f \to g$ given by components \[
    \eta_i\colon A^{\otimes i}\to \Hm{\B}({fA, hA})
    \] we want to construct an $A_\infty$-natural transformation $\chi_A(f)\to \chi_A(h)$. We define the component $\chi_A(\eta)_n(g^1, \ldots g^n)$ as \[ \sum_{i\geq n} \sum_{k_0+\ldots k_n=i-n}\eta_i(\delta^{\otimes k_0}\otimes g^1\otimes\delta^{\otimes k_1}\otimes g^2\otimes\ldots \otimes \delta^{\otimes k_{n-1}} \otimes g^n \otimes \delta^{\otimes k_n}).
    \]
    One verifies that $\chi_A$ commutes (strictly) with the composition of natural transformations and with the differentials, hence it defines a dg (and not $A_\infty$!) functor \[
     \chi_A\colon \Fun_{A_\infty}(A, \B)\to \Fun_{A_\infty}(D(A), \B).
    \]

\end{proof}
\subsection{The characteristic morphism}
As already noted, the formulas from the previous section are essentially the same appearing in \cite{lowchar}; let us make this fact precise. Let $C^\bullet(A)$ be the Hochschild cochain complex of $A$. It is well-known that there is an isomorphism \[
C^\bullet(A)\cong \Hm{\Fun_{A_\infty}(A,A)}(\id_A, \id_A),
\]where $A$ is viewed as a dg category with one object, and $\id_A$ is the identity  dg functor $A\tow{\id_A}A$. At the same time, there is a fully faithful embedding \[
\Fun_{A_\infty}(A,A)\hookrightarrow \Fun_{A_\infty}(A,D(A))
\] induced by the embedding $A\overset{\operatorname{inc}}{\hookrightarrow} D(A)$. The extension via $\chi_A$ of the inclusion $A\to D(A)$ is the identity functor of $D(A)$, hence the image of $\id_A$ via the composition \[
\Fun_{A_\infty}(A,A)\hookrightarrow \Fun_{A_\infty}(A,D(A)) \tow{\chi_A} \Fun_{A_\infty}(D(A),D(A))
\] is $\id_{D(A)}$. Therefore, the functor $\chi_A$ defines a map \[
C^\bullet(A)\cong \Hm{\Fun_{A_\infty}(A,A)}(\id_A, \id_A)\to \Hm{\Fun_{A_\infty}(D(A),D(A))}(\id_{D(A)}, \id_{D(A)})\cong C^\bullet(D(A)).
\] This coincides with the map $\chi_A$ defined in \cite{lowchar}.

\section{The 1-derived category}
In this section we show that, for a cdg deformation $A_\varepsilon$ of a dg algebra $A$, its filtered derived category $D^\varepsilon(A_\varepsilon)$ is a categorical deformation of $D(A)$, and the cocycle that it identifies corresponds to the one induced by the deformation $A_\varepsilon$ (Theorem \ref{classisclass}).
\subsection{Curved algebras and deformations}

We now briefly recall the notion of curved deformations and their 1-derived categories; for more details, see \cite{Nder, MoritacDef} and the references therein.

Let $R$ be a commutative ring. A cdg (curved differential graded) $R$-algebra $\A$ is given by a graded $R$-algebra $\A^\#$ together with a degree 1 derivation $d_\A\in \Hom_R(\A^\#,\A^\#)$ -- called (pre)differential -- and an element $c\in \A^2$ such that $d_\A(c)=0$ and $d^2_\A=[c, -]$ -- called curvature. 

A right cdg module over a cdg algebra $\A$ is a right graded $\A^\#$-module $M^\#$ equipped with a degree $1$ derivation $d_M\in \Hom_R(M^\#,M^\#)^1$ such that $d^2_M m= -mc$ for all $m\in M$; the category of right cdg $\A$-modules is in a natural way a dg category, that we'll denote with $\A\Mod$.

A cdg (first order) deformation of a dg algebra $A$ is given by a structure $A_\varepsilon$ of cdg $\ke$-algebra on the graded $\ke$-module $A\otimes_k\ke\cong A[t]/t^2A[t]$ which reduces mod $t$ to the dg algebra structure of $A$. Any curved deformation $A_\varepsilon$ corresponds to a certain Hochschild cocycle $\mu_A\in \mathbf{C}^2(A)$; write $\mu_i\in \Hm{k}(A^{\otimes i}, A)$ for its $i$-th component. 

The 1-derived category $D^\varepsilon(A_\varepsilon)$ of $A_\varepsilon$ is defined as the quotient of the homotopy category $H^0(A_\varepsilon\Mod)$ by the subcategory given by the modules for which the associated graded with respect to the $t$-adic filtration is acyclic.

 The goal of this section is to prove the following:
\begin{Thm}\label{classisclass}
    Let $A$ be a dg algebra, and $A_\varepsilon$ a cdg deformation of $A$ corresponding to a Hochschild cocycle $\mu_A\in \mathbf{C}^2(A)$. Then the category $D^\varepsilon (A_\varepsilon
    )$ is a categorical deformation of $D(A)$, and the class $\mu_{D(A)}(D^\varepsilon(A_\varepsilon
    ))\in \HH^2(D(A))$ coincides with $\chi_A(\mu_A)$.
\end{Thm}
Note that we have used the notation $\mu_A$ to denote both the cocycle in $\mathbf{C}^2(A)$ and the corresponding class in $\HH^2(A)$; this makes sense because the map $\chi_A$ is defined at the level of cocycles, and not only for classes.
\begin{Rem}
    Specifically, the above result implies (in fact, since we are operating at the level of classes is equivalent to) the fact that the image of $\mu_{D(A)}(D^\varepsilon(A_\varepsilon))$ via the restriction $\HH^2(D(A))\to \HH^2(A)$ coincides with $\mu_A$.
\end{Rem}
\subsection{Categorical deformations and the 1-derived category}

It is straightforward to see that $D^\varepsilon(A_\varepsilon)$ is a categorical deformation of $D(A)$. We know from \cite[Theorem 6.7]{Nder} that there is a recollement
\[\begin{tikzcd}
	D(A) & {D^\varepsilon(A_\varepsilon)} & D(A).
	\arrow["i", hook, from=1-1, to=1-2]
	\arrow["\Img t", from=1-2, to=1-3]
	\arrow["\Ker t", curve={height=-22pt}, from=1-2, to=1-1]
	\arrow["\Coker t"', curve={height=22pt}, from=1-2, to=1-1]
	\arrow["G"', curve={height=22pt}, from=1-3, to=1-2]
	\arrow[ curve={height=-22pt}, from=1-3, to=1-2]
\end{tikzcd}\] 

For concreteness, we will consider the dg enhancement of $D^\varepsilon(A_\varepsilon)$ given by the $A_\varepsilon$-modules which are cofibrant according to the model structure from \cite[Section 8]{Nder}; with this enhancement, the functors $i$, $\Coker t$, $\Ker t$, $\Img t$ are represented by honest dg functors while $G$ is given by a quasi-functor. The recollement induces a natural transformation $\Ker t \tow{\alpha} \Coker t$ which projects the subobject $\Ker t_M\subseteq M$ into the quotient $\Coker t_M \cong M/tM$. Moreover, there are natural transformations $\Img t\tow{\delta_1} \Ker t$ and $\Coker t\tow{\delta_2}\Img t$ induced respectively by the inclusion $tM\hookrightarrow \Ker t_M$ and the multiplication $M/tM \tow{t} tM$. For any $A_\varepsilon$-module $M$, the sequence \[
0\to tM \tow{\delta_1} \Ker t_M \tow{\alpha} \Coker t_M \tow{\delta_2} \Img tM \to 0
\] is exact. Splitting this yields short exact sequences  \[
0\to tM\tow{\delta_1} \Ker t_M \to \frac{\Ker t_M}{tM}\to 0
\text{ and }
0\to \frac{\Ker t_M}{tM}\to \Coker t_M \tow{\delta_2} tM \to 0
\] which are natural in $M$, and have the property that the composition \[\Ker t_M \to \frac{\Ker t_M}{tM}\to \Coker t_M\] coincides with $\alpha$. Since semifree $A$-modules are in particular projective as graded $A$-modules, the functor $\Hm{A}(N,-)$ is exact and we have two short exact sequences of $D^\varepsilon(A_\varepsilon)$-$D(A)$ bimodules  \[
0\to \Hm{A}(-,\Img t)\tow{\delta_1} \Hm{A}(-,\Ker t) \to \Hm{A}(-,\frac{\Ker t}{\Img t})\to 0  \]and \[ 0\to \Hm{A}(-,\frac{\Ker t}{\Img t})\to \Hm{A}(-,\Coker t) \tow{\delta_2} \Hm{A}(-,\Img t) \to 0
\] with the property that the composition \[\Hm{A}(-,\Ker t) \to \Hm{A}(-,\frac{\Ker t}{\Img t})\to \Hm{A}(-,\Coker t)\] coincides with $\alpha$. Since short exact sequences give triangles in the derived category, we are done. Note that the role of the functor $C$ from Definition \ref{defdef} is taken by the dg functor $\frac{\Ker t}{\Img t}$.

\begin{Rem}
    It follows from \cite[Proposition 7.10]{Nder} that the subcategory $D^{\operatorname{si}}(A_\varepsilon)\subseteq D^\varepsilon(A_\varepsilon)$ given by the semiderived category (\cite{Positselski_2018}) coincides with the kernel of the functor $C$ -- the cone of the natural transformation $\delta_1$. This observation allows to define, for any categorical deformation $\T_\varepsilon$, a full subcategory $\T^{\operatorname{si}}\subseteq\T_\varepsilon$ given by the kernel of the functor $C$ appearing in the deformation data. This subcategory measures -- at least in the compactly generated case -- how close the deformation is to admit a classical, or uncurved, representative. Specifically, in the case of an algebra deformation, it was observed in \cite{Nder} that as soon an object $M\in D^{\operatorname{si}}(A_\varepsilon)$ exists whose reduction $M/tM$ is a compact generator of $D(A)$, the deformation $A_\varepsilon$ is equivalent to an uncurved one. We expect a similar behavior to also appear in the general case.

    On the other hand, we were not able to find any characterization for the quotient $D(A_\varepsilon)$ of $D^\varepsilon(A_\varepsilon)$, a priori only defined in the case where $A_\varepsilon$ is uncurved, which only makes use of the deformation data. This remains a relevant question; for some related discussions, see Appendix \ref{appendix}.
\end{Rem}
\begin{Ex}
    Recall the example of the graded field discussed in \cite[Example 4.2]{Nder}; let $A$ be the algebra $A=k[u, u^{-1}]$ with deg $u=2$, and let $A_\varepsilon$ be the deformation induced by the Hochschild cocycle $u\in Z^2\mathbf{C}(A)$, i.e. the cdg algebra $(k[u, u^{-1}, t], 0, tu)$. It was shown in \cite{Nder} that the semiorthogonal decomposition of $D^\varepsilon(A_\varepsilon)$ is actually orthogonal. This is consistent with our theory: since the Hochschild cocycle $u$ is an isomorphism, the same holds for the cocycle in $\chi_A(u)\in Z^2\mathbf{C}(D(A))$. Hence, its cone -- the gluing functor -- is the zero object.
\end{Ex}
\subsection{Enhancing the left adjoint}\label{leftadj}

As mentioned before, the functor $G$ is not induced by a dg functor. It follows essentially by definition that it does have an enhancement as a quasi-functor, but this is impractical: concretely, the natural formula must be resolved in order to do concrete computations and this creates complications. Instead, we will see that $G$ has a very natural incarnation as an $A_\infty$-functor. We refer the reader to \cite{Ainffunct, COS2} for details regarding the switch between different models for the homotopy category of dg categories. Note that we use for $A_\infty$-functors the same notion of adjunction  we used for quasi-functors; this is also equivalent (at the homotopy level) to the notion of adjunction between $\infty$-functors, see \cite{dyckerhoffsphericalrelative}.

We will consider a variant of the construction from \cite{Nder} of one of the compact generators of the $1$-derived category. Recall that the curvature of $A_\varepsilon$ is of the form $t\mu_0 $. We have a well-defined degree $2$ closed $A$-module map $
A \tow{t\mu_0} A_\varepsilon.
$  We define the $A_\varepsilon$-module $\Gamma$ as the ``two sided cone'' of the diagram 
\[\begin{tikzcd}
	A & {A_\varepsilon}
	\arrow["-t\mu_0", shift left, from=1-1, to=1-2]
	\arrow["\pi", shift left, from=1-2, to=1-1]
\end{tikzcd}\] where $\pi$ is the natural projection. Explicitly, $\Gamma$ is given by the graded module $A_\varepsilon\oplus A[-1]$ with differential given by \[
d_\Gamma(a_\varepsilon, b)=(d_{A_\varepsilon}a_\varepsilon - t\mu_0 b, d_{A[-1]}b+a), 
\] where we have denoted the action of the map $\pi$ by removing the subscript $\varepsilon$.
\begin{Prop}
    The module $\Gamma$ is a cdg $A_\varepsilon$-module.
\end{Prop}
\begin{proof}
    Straightforward, see \cite[Proposition 4.2]{Nder}.
\end{proof}
Since $t\Gamma\cong A$, there is a natural dg algebra map \[
E=\Hm{A_\varepsilon}(\Gamma, \Gamma)\to \Hm{A}(t\Gamma, t\Gamma)\cong A
\] which is easily seen to be a surjective quasi-isomorphism (see \cite[Proposition 6.5]{Nder}). Hence, it admits an $A_\infty$-inverse $e\colon A \to E$. Let us describe explicitly this inverse, remembering that the $A_\infty$-algebra morphism $e$ is given by a collection \[
g_i\in \Hm{k}(A^{\otimes i}, E)
\] of degree $1-i$ satisfying the $A_\infty$-identities\[
    \sum_{n=r+s+t} (-1)^{r+st}g_u(\id^{\otimes r}\otimes m_s \otimes \id^{\otimes r})=\sum_{i_1+\ldots i_r=n}(-1)^s m_r(g_{i_1}\otimes g_{i_2}\otimes \ldots  \otimes g_{i_r})
\] where $m_i$ represent the operations of $A$ and $E$, and  $s=\sum_j (r-j)(i_j-1)$. At the graded level, the algebra $E$ is the matrix algebra \[
\begin{bmatrix}
\Hm{A_\varepsilon}(A_\varepsilon, A_\varepsilon) & \Hm{A_\varepsilon}(A[-1], A_\varepsilon) \\
\Hm{A_\varepsilon}(A_\varepsilon, A[-1]) & \Hm{A_\varepsilon}(A,A)
\end{bmatrix}\cong\begin{bmatrix}
A_\varepsilon & A[1] \\
A[-1] & A
\end{bmatrix}
\] with differential \[
d_E\begin{bmatrix}
x_\varepsilon & y \\
z & w
\end{bmatrix} =
\begin{bmatrix}
d_\varepsilon x_\varepsilon & d y \\
dz & dw
\end{bmatrix}+ \begin{bmatrix}
-t(y+\mu_0z) & \mu_0w-x\mu_0 \\
x-w & 0
\end{bmatrix}
\] for degree $0$ elements $x_\varepsilon, y, z, w$ -- otherwise the Koszul sign rule applies. The component $g_1$ carries the element $a\in A$ to the matrix  \[
\begin{bmatrix}
a & \mu_1(A) \\
0 & a
\end{bmatrix};
\]this does commute with the differentials, but not with the multiplications -- since the map $A\tow{1\to 1}A_\varepsilon$ is not a map of associative algebras, neither is $g_1$. A higher component $
g_2\colon A\otimes_k A \to E
$ is therefore needed. We define \[
g_2(a, b)= \begin{bmatrix}
0 & \mu_2(a,b) \\
0 & 0
\end{bmatrix}.
\]
\begin{Prop}
The map $A\tow{g}E$ is a quasi-isomorphism of $A_\infty$-algebras.
\end{Prop}
\begin{proof}
    By definition, $g$ is a quasi-isomorphism if and only if $g_1$ is; but $g_1$ is a right inverse to the quasi-isomorphism $E\to A$, so the only thing to check is that $g$ is indeed an $A_\infty$-algebra morphism. Since $A$ and $E$ are dg algebras and $g_i=0$ for $i>2$, the $A_\infty$ equations specialize to \begin{align}
        &d_Eg_1(a)=g_1d_A(a) \label{coc1}\\
        &g_1(ab)-g_1(a)g_1(b)=d_E g_2(a, b)+g_2(d_Aa, b)+g_2(a, d_Ab)\label{coc2}\\
        &  \label{coc3}      g_2(a,b)g_1(c)-g_1(a)g_2(a,b)+g_2(ab,c)-g_2(a, bc)=0
        \\
        & \label{coc4}
        g_2(a,b)g_2(b,c)=0
        \end{align}
    for degree $0$ elements $a,b,c\in A$ -- otherwise the Koszul sign rule must be applied.
    Writing down the explicit formulas, one sees that equation \eqref{coc1} corresponds to the condition $d\mu_1+\mu_1d=[\mu_0,-]$, equation \eqref{coc2} to \[
    a\mu_1(b)-\mu_1(ab)+\mu_1(a)b=d\mu_2(a,b)-\mu_2(a, db)-\mu_2(a, db)
    \] and equation \eqref{coc3} to \[
    a\mu_2(b,c)-\mu_2(ab,c)+\mu_2(a, bc)-\mu_2(a, b)c.
    \]
    Equation \eqref{coc4} is automatically satisfied, since $g_2$ is a square-zero matrix. There are precisely (all but one of) the components of the equation $d_H(\mu)=0$, where $d_H$ is the total Hochschild differential. Note how the only condition that was not needed in this proof -- that is, $d_A \mu_0=0$ -- was implicitly used in showing that $\Gamma$ is a cdg $A_\varepsilon$-module.
\end{proof}
Then, seeing $A$ as a one-object dg category, we have a well-defined $A_\infty$-functor \[
A\tow{g} D^\varepsilon(A_\varepsilon)
\] which sends the only object $A$ to $\Gamma\in D^\varepsilon(A_\varepsilon)$. 

\subsubsection{Cones of $A_\infty$-morphisms}
Let $F, G\colon A\to D(A)$ be (strict) $A_\infty$-functors and $\eta\colon F\to G$ an $A_\infty$-natural transformation. In components, $F$ is given by a certain object $FA\in D(A)$ together with an $A_\infty$-algebra map \[
A\tow{F} \Hm{A}(FA, FA)
\]  which is given explicitly by components \[
F_i\colon A^{\otimes i} \to \Hm{A}(FA, FA)
\] of appropriate degree, and the same for $G$. Similarly, the natural transformation $\eta$ is given by components \[
\eta_i\colon A^{\otimes i} \to \Hm{A}(FA, GA).
\] The category of $A_\infty$-functors $A\to D(A)$ is pretriangulated, and we can give an explicit description of cones in this category, following \cite{lefèvrehasegawa2003surlesainfinicategories}. At the object level, the $A_\infty$-functor $\operatorname{Cone}(\eta)(A)$ is defined as \[
\operatorname{Cone}(\eta)(A)=\operatorname{Cone}(\eta_0)\cong_{gr} GA\oplus FA[1].
\] The $A_\infty$-morphism \[
A\to \Hm{A}(\operatorname{Cone}(\eta_0), \operatorname{Cone}(\eta_0))\cong\Hm{A}(GA\oplus FA[1], GA\oplus FA[1])
\] is given in components by  \[
   \begin{split}
       \begin{bmatrix}
F_i & \eta_i \\
0 & G_i
\end{bmatrix}\colon A^{\otimes i} &\to \Hm{A}(GA\oplus FA[1], GA\oplus FA[1]).
   \end{split}
    \]

    Any such morphism defines a canonical short exact sequence of $A_\infty$-functors \[
0 \to GA \to \operatorname{Cone}(\eta) \to FA[1]\to 0
\] which corresponds to the triangle \[
FA \tow{\eta} GA \to \operatorname{Cone}(\eta) \to FA[1] 
\] in the homotopy category.

\begin{Lem}\label{cones}
    The composition \[
    A \tow{g} D^\varepsilon(A_\varepsilon)\tow{\Ker t} D(A)
    \] is isomorphic to the cone of $\mu_A$, when seen as a closed morphism $\operatorname{inc}[-2]\to \operatorname{inc}$. Moreover the composition \[
    A \tow{g} D^\varepsilon(A_\varepsilon)\tow{\Coker t} D(A)
    \] coincides with the cone of the identity natural transformation of $\operatorname{inc}[-1]$.
\end{Lem}

\begin{proof}
    For the first statement one observes that, at the graded level, $\Ker t_\Gamma=A\oplus A[-1]$, with differential induced by $\Gamma$. Under the isomorphism $\Ker A_\varepsilon\cong A$, we see that \[
    d_{\Ker t_\Gamma}(a, b)=(da, db+a\mu_0)
    \] which shows that $\Ker t_\Gamma$ corresponds to the cone of $\mu_0$. Hence, as an object of $D(A)$, $\Ker t_\Gamma$ coincides with the cone of $\mu$. To conclude, we have to show that this also holds at the level of the actions, i.e. that the two possible $A_\infty$-algebra maps \[
    A\to\Hm{A}(\Ker t_\Gamma, \Ker t_\Gamma)
    \] -- one defined by seeing $\Ker t_\Gamma$ as the cone of $\mu_A$, the other induced by $e$ -- coincide. This is however immediate, since they are both given in components by \[
   \begin{split}
       A &\to \Hm{A}(\Ker t_\Gamma, \Ker t_\Gamma)\\
       a &\to  \begin{bmatrix}
a & \mu_1(a) \\
0 & a
\end{bmatrix}
   \end{split}
    \]
    and 
    \[
   \begin{split}
       A\otimes A &\to \Hm{A}(\Ker t_\Gamma, \Ker t_\Gamma)\\
       a\otimes b &\to  \begin{bmatrix}
0 & \mu_2(a, b) \\
0 & 0
\end{bmatrix}.
   \end{split}
    \]
    The second statement is straightforward using the same argument.
\end{proof}
\subsubsection{The adjunction isomorphism}
Using Proposition \ref{propextend}, we can extend the functor $g$ to an $A_\infty$-functor \[
G_\infty:= \chi_A(g)\colon D(A)\to D^\varepsilon(A_\varepsilon).
\]

Given $M=(\{n_i\}, f )\in \operatorname{Tw}(A)$, we have a natural isomorphism \[
tG_\infty M \cong (\oplus_i t\Gamma[n_i], g_1(f)_{|t\Gamma})\cong (\oplus_i A[n_i], f)\cong M
\] We therefore have a natural isomorphism $\id_{D(A)}\to \Img t\circ G_\infty$. 
This induces a natural transformation \[
\Hm{A_\varepsilon}(G_\infty M,  N)\to \Hm{A}(tG_\infty M, t N)\to \Hm{A}(M, t N).
\]
We can show that this emerges as the unit of an adjunction between $\Img t$ and $G_\infty$.
\begin{Lem}
    The natural morphism \begin{equation}\label{unitadj}
    \Hm{A_\varepsilon}(G_\infty M,  N)\cong \Hm{A}(M, tN)
    \end{equation} is an isomorphism in $D(k)$.
\end{Lem}

\begin{proof}
Consider the subcategory of $D(A)$ given by the modules for which \eqref{unitadj} is an isomorphism. This is closed under shifts and cones, as well as under coproducts since $G_\infty$ preserves them. Hence it is enough to show that it contains the generator $A$; Since $G_\infty A=\Gamma$, we want to show that the map \[
\begin{split}
    \Hm{A_\varepsilon}(\Gamma, M)\to \Hm{A}(t\Gamma, tM)\cong \Hm{A}(A, tM)\cong tM
\end{split}
\] is a quasi-isomorphism. It is immediate to check that it is surjective, and its kernel is given by \[
\Hm{A_\varepsilon}(\Gamma, \Ker t_M)\cong \Hm{A}(Q\Gamma, \Ker t_M)
\] which is acyclic since $Q\Gamma$ is contractible.
\end{proof}
Hence by \cite[Proposition 1.2.7]{dyckerhoffsphericalrelative} the functor $G_\infty$ is left adjoint to $\Img t$.

We'll also need the ``large'' version of Lemma \ref{cones}.
\begin{Lem}
    The composition \[
    D(A) \tow{G_\infty} D^\varepsilon(A_\varepsilon)\tow{\Ker t} D(A)
    \] is isomorphic to the cone of $\chi_A(\mu_A)$, when seen as a closed morphism $\id_{D(A)}[-2]\to \id_{D(A)}$. Moreover the composition \[
    D(A) \tow{G_\infty} D^\varepsilon(A_\varepsilon)\tow{\Coker t} D(A)
    \] coincides with the cone of the identity natural transformation of $\id_{D(A)}[-1]$.
\end{Lem}
\begin{proof}
    This is a combination of Lemma \ref{cones} and the properties of the functor $\chi_A$.
\end{proof}
\begin{proof}[Proof of Theorem \ref{classisclass}]
First of all, observe that the unit map $\Img t\circ G_\infty \to \id_{D(A)}$ is the identity. Thus by definition, the class $\mu_{D(A)}(D^\varepsilon(A_\varepsilon))$ is obtained by composing with $G_\infty$ the morphism\[
\Img t \to\Img t[2]
\] obtained by composing the boundary maps induced by the two short exact sequences 
\[
0 \to \frac{\Ker t}{\Img t} \to \Coker t \tow{\delta_2} \Img t \to 0
\] and \[
0 \to \Img t \tow{\delta_1} \Ker t \to \frac{\Ker t}{\Img t} \to 0.
\] The comparison between quasi-functors and $A_\infty$-functors preserves dg functors and dg natural transformations between them, hence the class $\mu_{D(A)}(D^\varepsilon(A_\varepsilon))$ is given, in the $A_\infty$ setting, by the composition of the morphism \[
\id_{A}\to \frac{\Ker t}{\Img t}\circ G_\infty [1]
\] induced by the short exact sequence \begin{equation}\label{sesainf1}
0 \to  \frac{\Ker t}{\Img t} \circ G_\infty \to \Coker t \circ G_\infty \tow{\delta_2 G}   \Img t \circ G\cong \id_{A}\to 0
\end{equation}
 with the morphism \[
\frac{\Ker t}{\Img t}\circ G_\infty[1] \to \id_{D(A)}[2] 
\]
induced by the short exact sequence \begin{equation}\label{sesainf2}
0 \to  \id_{D(A)}\cong \Img t \circ G_\infty \tow{\delta_1} \Ker t\circ G_\infty \to \frac{\Ker t}{\Img t} \circ G_\infty    \to 0.
\end{equation}The first observation is that $\frac{\Ker t}{\Img t}\circ G_\infty$ coincides with $\id_{D(A)}[-1]$; moreover by Proposition \ref{propextend} the short exact sequence \eqref{sesainf1} can be read as \[
0 \to \id_{D(A)}[-1]\to \operatorname{Cone}(\id_{\id_{D(A)}[-1]}) \to \id_{D(A)} \to 0
\] from which\footnote{The confusing notation $\id_{\id_{D(A)}[-1]}$ denotes the identity natural transformation between the shifted identity functor and itself.} one can conclude that the boundary map is the identity map $\id_{D(A)}\to \id_{D(A)}$. In the same way, the short exact sequence \eqref{sesainf2} reads \[
0 \to \id_{D(A)}[1]\to \operatorname{Cone}(\chi_A(\mu_A)) \to \id_{D(A)} \to 0
\] from which it follows that the boundary map is precisely $\chi_A(\mu_A)$, hence the claim.
    
\end{proof}
\section{A commutative square of deformations}
Finally, we can show that the map $\mu_{D(A)}$ is compatible with the bijection \[
\nu \colon \cdef_A(\ke)\to \HH^2(A)
\]introduced in \cite{MoritacDef}, yielding the square \eqref{squarec} promised in the introduction (Theorem \ref{mainth}). 
\subsection{Curved Morita deformations and Hochschild cohomology}

We begin by recalling some notions from \cite{MoritacDef}. Recall that if $A, B$ are dg algebras, an $A$-$B$-bimodule $X$ is said to be a Morita equivalence if the induced adjoint pair 
\[\begin{tikzcd}
	{D(A)} & {D(B)}
	\arrow["{-\otimes_A^{\operatorname{L}} X}", shift left, from=1-1, to=1-2]
	\arrow["{\mathbb{R}\Hm{B}(X,-)}", shift left, from=1-2, to=1-1]
\end{tikzcd}\] is an equivalence. For notational simplicity, denote with $F_X$ the equivalence $-\otimes_A^{\operatorname{L}}X$ and with $F_X^{-1}$ its inverse $\mathbb{R}\Hm{B}(X,-)$.

Let $A$ be a dg algebra. A curved Morita deformation of $A$ is a $\ke$-free cdg $\ke$-algebra $B_\varepsilon$ equipped with, setting $B:=B_\varepsilon\otimes_{\ke}k$, a $B$-$A$ Morita bimodule $X$. Two curved deformations $B_\varepsilon$, $C_\varepsilon$ are considered equivalent if there exists an appropriately cofibrant $B_\varepsilon$-$C_\varepsilon$ bimodule $X_\varepsilon$ which commutes with the equivalences of the bases with $A$; the set $\cdef_A(\ke)$ is defined as the set of (first order) equivalence classes of curved Morita deformations of $A$. For a given curved Morita deformation $(B_\varepsilon, X)$ of $A$, by the main result of \cite{KellerInvariance} the bimodule $X$ induces an isomorphism \[
\varphi_X\colon \HH^2(B)\to \HH^2(A).
\]

A map   \[
\nu \colon \cdef_A(\ke)\to \HH^2(A);
\] was constructed in \cite{MoritacDef} in the following way: the $\ke$-free algebra $B_\varepsilon$ defines an cdg deformation of its reduction $B$, and as such defines a class $\mu_B\in \HH^2(B)$; the class $\nu(B_\varepsilon)$ is then defined as $\varphi_X(\mu_B)\in \HH^2(A)$. It was shown in \cite{MoritacDef} that this induces a bijection after passing to the quotient. 

\subsection{From algebras to categories}

We start with the following:
\begin{Prop}
    For any curved Morita deformation $(B_\varepsilon, X)$ of $A$, the 1-derived category $D^\varepsilon(B_\varepsilon)$ is a categorical deformation of $D(A)$. Equivalent Morita deformations yield equivalent categorical deformations, so this assignment gives a well-defined map \[
    D^\varepsilon(-)\colon \cdef_A(\ke)\to \CatDef_{D(A)}(\ke)
    \]
\end{Prop}

\begin{proof}
    By Theorem \ref{classisclass}, the category $D^\varepsilon(A_\varepsilon)$ is a categorical deformation of $D(B)$, i.e. there exists a recollement 
\[\begin{tikzcd}
	{D(B)} & {D^\varepsilon(B_\varepsilon)} & {D(B)}
	\arrow["i"{description}, from=1-1, to=1-2]
	\arrow["K", shift left=3, from=1-2, to=1-1]
	\arrow["Q"', shift right=3, from=1-2, to=1-1]
	\arrow["\I"{description}, from=1-2, to=1-3]
	\arrow[shift left=3, from=1-3, to=1-2]
	\arrow["G"', shift right=3, from=1-3, to=1-2]
\end{tikzcd}\]
together with two triangles  \begin{equation}\label{trianglesYoneda}\I  \to K \to C \to \I[1]\text{ and }C\to  Q \to \I \to C[1].\end{equation} 

We can thus define \[
i_A=iF_X\text{, }G_A=GF_X\text{, } \I_A=F_X^{-1}\I \text{, }K_A=F_X^{-1}L\text{, }Q_A=F_X^{-1}Q
\] to obtain a recollement
\[\begin{tikzcd}
	{D(A)} & {D^\varepsilon(B_\varepsilon)} & {D(A);}
	\arrow["{i_A}"{description}, from=1-1, to=1-2]
	\arrow["{K_A}", shift left=3, from=1-2, to=1-1]
	\arrow["{Q_A}"', shift right=3, from=1-2, to=1-1]
	\arrow["{\I_A}"{description}, from=1-2, to=1-3]
	\arrow[shift left=3, from=1-3, to=1-2]
	\arrow["{G_A}"', shift right=3, from=1-3, to=1-2]
\end{tikzcd}\] similarly one can compose the triangles \eqref{trianglesYoneda} with $F_X^{-1}$ to obtain the relevant Yoneda extension, and thus show that $D^\varepsilon(B_\varepsilon)$ is indeed a deformation of $D(A)$.

To show that equivalent curved deformations yield equivalent categorical deformations, assume that $Z_\varepsilon$ is an equivalence between two Morita deformations $B_\varepsilon$ and $C_\varepsilon$. Then by \cite[Proposition 3.6]{MoritacDef} there is an induced equivalence $D^\varepsilon(B_\varepsilon)\cong D^\varepsilon(C_\varepsilon)$ which, by \cite[Lemma 3.5]{MoritacDef} is compatible with the semiorthogonal decompositions; it is straightforward to check that this is also an equivalence of deformations -- this essentially follows from the fact that the equivalences are $\ke$-linear and the Yoneda extensions are defined in terms of the action of $t$. 
\end{proof}

We can now tie everything together: 
\begin{Thm}\label{mainth}
    Let $A$ be a dg algebra. There is a commutative square of bijections
\[\begin{tikzcd}
	{\cdef_A(\ke)} & {\HH^2(A)} \\
	{\CatDef_{D(A)}(\ke)} & {\HH^2(D(A))}
	\arrow["\nu", from=1-1, to=1-2]
	\arrow["{D^\varepsilon(-)}"', from=1-1, to=2-1]
	\arrow["{\chi_A}", from=1-2, to=2-2]
	\arrow["\mu_{D(A)}", from=2-1, to=2-2]
\end{tikzcd}\]
    \end{Thm}
\begin{proof}
    We already know the arrows $\chi_A$, $\mu_{D(A)}$ and $\nu$ are bijections, so only 
    the commutativity of the diagram remains to be shown. Let $(B_\varepsilon, X)$ be a curved Morita deformation of $A$; then $D^\varepsilon(B_\varepsilon)$ is a categorical deformation of $D(A)$, and we want to compute its associated class in $\HH^2(D(A))$. The equivalence $F_X$ defines a bijection $\HH^2(D(B))\tow{\varphi_{F_X}} \HH^2(D(A))$: given a class \[
    \id_{D(B)}[-1]\tow{\eta_B}\id_{D(B)}[1]
    \] in $\HH^2(D(B))$, we define $\varphi_{F_X}(\eta_B)$ as the class \[
    \id_{D(A)}[-1]\cong F_X^{-1}\id_BF_X[-1]\tow{F_X^{-1}\eta_B F_X} F_X^{-1}\id_BF_X[-1]\cong \id_{D(A)}[-1]
    \] in $\HH^2(D(A))$;  this is compatible with the morphism $\chi_A$, in the sense that the diagram 
\begin{equation*}\label{HHcomm}
\begin{tikzcd}
	{\HH^2(B)} & {\HH^2(A)} \\
	{\HH^2(D(B))} & {\HH^2(D(B))}
	\arrow["{\varphi_X}", from=1-1, to=1-2]
	\arrow["{\chi_B}"', from=1-1, to=2-1]
	\arrow["{\chi_A}", from=1-2, to=2-2]
	\arrow["{\varphi_{F_X}}", from=2-1, to=2-2]
\end{tikzcd}\end{equation*}
 commutes. Recall that $\mu_B\in \HH^2(B)$ is the class corresponding to the cdg deformation $B_\varepsilon$ of $B$. By construction, one has that \[\mu_{D(A)}(D^\varepsilon(B_\varepsilon))=\varphi_{F_X}\circ \mu_{D(B)}(D^\varepsilon(B_\varepsilon)).\] Hence, by Theorem \ref{classisclass}, we get \[
 \mu_{D(A)}(D^\varepsilon(B_\varepsilon))=\varphi_{F_X}(\chi_B(\mu_B)){=}\chi_A(\varphi_X(\mu_B))=\chi_A(\nu(B_\varepsilon))
 \] and we are done.
\end{proof}

\section{Addendum: coherent complexes and higher actions}\label{coherentstuff}
As already discussed, our definition of categorical deformation only requires working with the homotopy categories of the various functor categories, without considering their enhancements. This allowed us to give fairly direct constructions, but also has its downsides. Namely, our theory is well-equipped to decide whether two deformations are equivalent, but less so to describe the group of autoequivalences of a deformation. A notion of morphism of deformation more apt to this problem would entail having a functor which induces in an appropriate sense a morphism of Yoneda extensions (see Proposition \ref{morphext}). The issue is that, in our setting, it is unreasonable for a functor to commute with the boundary object $C$, since that is often only defined up to a noncanonical isomorphism. For this construction to work, one would need a fully enhanced notion of categorical deformation.

Indeed, one would want to define a Yoneda extension of functors as a sequence \[
0\to \I \tow{\delta_1}K \tow{\alpha}Q \tow{\delta_2}\I  \to 0
\] with $\alpha\delta_1=\delta_2\alpha$ with the property that $\alpha$ induces an isomorphism $\Bar{\alpha}$ between the cone $C$ of $\delta_1$ and the cocone $D$ of $\delta_2$, as in \[\begin{tikzcd}[ampersand replacement=\&]
	\I \& K \& C \\
	D \& Q \& \I.
	\arrow["{\delta_1}", from=1-1, to=1-2]
	\arrow[from=1-2, to=1-3]
	\arrow["\alpha"'{pos=0.1}, from=1-2, to=2-2]
	\arrow["{\bar{\alpha}}"'{pos=0.2}, out=275, in=100, dashed, from=1-3, to=2-1]
	\arrow[from=2-1, to=2-2]
	\arrow["{\delta_2}", from=2-2, to=2-3]
\end{tikzcd}\] Even though one can show that $\alpha\delta_1=\delta_2\alpha$ the issue is that, if we remain in the homotopy category, the Toda bracket $\langle \delta_1, \alpha, \delta_2 \rangle$ is an obstruction to the existence of any map $\Bar{\alpha}$ as in the diagram. This problem can be solved by committing to the higher categorical world; indeed, sequences with vanishing Toda Brackets are known \cite{ariottacoh} to correspond to \emph{coherent complexes}. To be more precise, one can consider the data of a recollement, together with the following morphisms in the (enhanced) functor category:
\begin{itemize}
    \item A sequence of morphisms \begin{equation}\label{enhancedextension}
       0\to \I \tow{\delta_1}K\tow{\alpha} Q \tow{\delta_2} \I \to 0;
         \end{equation}
    \item A nullhomotopy $t$ for the composition $\alpha\delta_1$ and a nullhomotopy $s$ for $\delta_2\alpha$;
    \item A homotopy $H$ between $\delta_2t$ and $s\delta_1$.
\end{itemize}
This data defines a canonical map $\Bar{\alpha}$ between the cone of $\delta_1$ and the cocone of $\delta_2$; a categorical deformation would then be given by the above data, under the condition for $\Bar{\alpha}$ to be an equivalence; note that this is a way to express the ``exactness'' of the sequence \eqref{enhancedextension}. Any such construction yields, by passing to the homotopy category, a Yoneda extension in the sense of Section \ref{Yoneda}. Highlighting the role of the natural transformations $\delta_1$ and $\delta_2$ also hints at the sense in which the category $\T_\varepsilon$ ought to be considered $\ke$-linear. Indeed, under the adjunctions between $i$ and $K$ and $Q$, the two natural transformations \[
\I \tow{\delta_2}K \text{ and } Q \tow{\delta_2} E
\] correspond to natural transformations \[
i\I \to \id_{\T_\varepsilon}\to iE.
\] This should be though as the categorification of the two natural transformations \[
tM \hookrightarrow M\tow{t} tM
\] where $M$ is any $\ke$-module; hence, $\ke$ ``acts'' on the category $\T_\varepsilon$ via the action of the endofunctor $i\I$, with the various natural transformations acting as compatibility conditions. 

In order to work effectively with these objects, however, our framework 
has to be recast in the language of $\infty$-category theory;  a full treatment of these, more delicate, higher aspects is thus postponed to future work. 
Such a homotopical treatment of the theory will also be one of the necessary ingredients in showing that, when appropriately generalized, categorical deformations are \emph{the} formal moduli problem associated to the Hochschild complex. The results in this paper should indeed be thought of as describing -- via first order categorical deformations -- the tangent space to this deformation problem. 

We conclude by observing that further, leaving the model of dg categories would allow to give a fairly straightforward description of \emph{absolute}, or non-linear deformations in terms of topological Hochschild cohomology, which should in turn be compared with \cite[Section 5]{KaledinLowen}\cite{KaledinLowenBooth2, KaledinLowenBooth1}.

\appendix
\section{Deformations and smoothness}\label{appendix}
As an application, we show that, unlike in the classical case, smoothness is preserved under our new type of deformations. Recall that a dg category $\T$ is said to be smooth if the identity bimodule $\id_\T\in D(\op{\T}\otimes \T)$ is perfect. \begin{Prop}\label{defosmooth}
    Let $\T_\varepsilon$ be a categorical deformation of a triangulated category $\T$. Then $\T_\varepsilon$ is smooth if and only if $\T$ is smooth; the same statement holds for properness.
\end{Prop}
\begin{proof}
    One implication follows directly from \cite[Proposition 4.9]{Kuznetsov_2014}, since $\T$ is the only semiorthogonal factor of $\T_\varepsilon$. For the other implication, we know again by \cite[Proposition 4.9]{Kuznetsov_2014} that, being $\T$ smooth, the gluing $\T_\varepsilon$ is smooth if and only if the gluing bimodule $KG[1]\in D(\op{\T}\otimes \T)$ is perfect. Since $\T$ is smooth, the diagonal bimodule $\id_\T\in D(\op{\T}\otimes \T)$ is perfect. From the existence of the triangle \[
    \id_\T[-1]\tow{\mu_\T(\T_\varepsilon)}\id_\T[1] \to KG[1] \to \id_\T
    \] of Corollary \ref{triacor} we deduce that $KG[1]$ is also perfect. The same argument holds for properness.
\end{proof}
In fact, the proof says something more: the functor $KG$ is obtained as a finite extension of copies of $\id_\T$, so it is ``as perfect as'' $\id_\T$. This, in turn, conveys the idea that $\T_\varepsilon$ is ``as singular as'' $\T$. From this, we obtain a more conceptual proof of \cite[Proposition 6.8]{Nder}:
\begin{Prop}
    If $A$ is a homologically smooth dg algebra and the deformation $A_\varepsilon$ is uncurved (i.e. a dg algebra), then $D^\varepsilon(A_\varepsilon)$ is a categorical resolution of $D(A_\varepsilon)$. 
\end{Prop}
\begin{proof}
    Fist of all, observe that since $A$ is homologically smooth the same holds for $D(A)$. It is shown in \cite[Corollary 3.10]{Nder} that there is an embedding \[
    D(A_\varepsilon) \hookrightarrow D^\varepsilon(A_\varepsilon);
    \] since $D^\varepsilon(A_\varepsilon)$ is a categorical deformation of $D(A)$, by Proposition \ref{defosmooth} it is homologically smooth and hence a categorical resolution of $D(A_\varepsilon)$.
\end{proof}

It is reasonable to see $D^\varepsilon(A_\varepsilon)$ as a \emph{blowup} of $D(A_\varepsilon)$. Indeed, the procedure of moving from $D(A_\varepsilon)$ to $D^\varepsilon(A_\varepsilon)$ only resolves the singularity coming from the presence of the nilpotent deformation parameter; any singularity ``away from $0$'' will still be present in $D^\varepsilon(A_\varepsilon)$. This procedure offers some insight into the relation between the classical notion of deformation for triangulated categories (\cite{DAGX, Tst3, Blanc_Katzarkov_Pandit_2018}) and our categorical deformations. The point is that, given a classical deformation, its blowup (in an appropriate sense) is a categorical deformation. One can use this setup to investigate the natural question of whether -- in specific cases -- classical deformations span the whole Hochschild complex and, in cases where they don't, which part of the Hochschild complex they do span. This question can be reformulated as asking which classes of categorical deformations can be obtained as blowups of classical deformations. This is useful in practice, since it allows one access to properties and invariants of the categorical deformation. In future work we will further explore this perspective, in particular in relation to the results from \cite{Tst3}.

\printbibliography

\end{document}